\documentclass[a4paper,12pt,twoside]{amsart}

\usepackage{color}
\usepackage{verbatim}
\usepackage{geometry}
\usepackage{float}
\usepackage{amsmath}
\usepackage{amsthm}
\usepackage{amssymb}
\usepackage{graphicx}
\usepackage{booktabs}
\usepackage{array}
\usepackage{MnSymbol}

\usepackage{caption}
\numberwithin{equation}{section}

\global\long\def\R{\mathbb{R}}%

\global\long\def\Rn{\mathbb{R}^{n}}%
\global\long\def\Rd{\mathbb{R}^{d}}%

\global\long\def\cP{\mathcal{P}}%
\global\long\def\cR{\mathcal{R}}%

\global\long\def\bQ{\boldsymbol{Q}}%
\global\long\def\cS{\mathcal{S}}%
\global\long\def\cE{\mathcal{E}}%
\global\long\def\rmD{\mathrm{D}}%

\global\long\def\fE{\mathfrak{E}}%

\global\long\def\bnu{\boldsymbol{\nu}}%

\global\long\def\norm#1{\left\Vert #1\right\Vert }%

\global\long\def\diver{\mathrm{div}}%
\global\long\def\d{\mathrm{d}}%
\global\long\def\diam{\mathrm{diam}}%
\global\long\def\e{\mathrm{{e}}}%

\global\long\def\jump#1{\lsem#1\rsem}%

\global\long\def\cL{\mathcal{L}}%
\global\long\def\cF{\mathcal{F}}%

\global\long\def\cT{\mathcal{T}}%

\global\long\def\S{\mathbb{S}}%

\newtheorem{thm}{Theorem}[section]
\newtheorem{cor}[thm]{Corollary}
\newtheorem{lem}[thm]{Lemma}
\newtheorem{prop}[thm]{Proposition}

\theoremstyle{definition}
\newtheorem{defn}[thm]{Definition}
\newtheorem{rem}[thm]{Remark}
\newtheorem{example}[thm]{Example}

\definecolor{rot}{rgb}{1.000,0.000,0.000}

\begin{document}
\author[M. Heida, M. Kantner, A. Stephan]{Martin Heida, Markus Kantner, Artur Stephan 
}
\address{Weierstrass Institute,
Mohrenstr. 39, 10117 Berlin, Germany }
\email{martin.heida@wias-berlin.de}
\email{markus.kantner@wias-berlin.de}
\email{artur.stephan@wias-berlin.de}
\title[Discretization for the Fokker--Planck operator]{Consistency and convergence for a family of finite volume discretizations
of the Fokker--Planck operator}	
\thanks{M. H. and A. S. are financed by Deutsche Forschungsgemeinschaft (DFG) through Grant
CRC 1114
``Scaling Cascades in Complex Systems'', Project C05 {\em Effective models for
materials and interfaces with multiple scales}. The work of M. K. received funding from the DFG under
Germany's Excellence Strategy -- EXC2046: \textsc{Math}+ (Berlin Mathematics Research Center).
}	
\maketitle
\begin{abstract}
We introduce a family of various finite volume discretization schemes
for the Fokker--Planck operator, which are characterized by different
weight functions on the edges. This family particularly includes the
well-established Scharfetter--Gummel discretization as well as the
recently developed square-root approximation (SQRA) scheme. We motivate
this family of discretizations both from the numerical and the modeling
point of view and provide a uniform consistency and error analysis.
Our main results state that the convergence order primarily depends
on the quality of the mesh and in second place on the quality of the
weights. We show by numerical experiments that for small gradients
the choice of the optimal representative of the discretization family
is highly non-trivial while for large gradients the Scharfetter--Gummel
scheme stands out compared to the others.

\end{abstract}
\section{Introduction}

The Fokker--Planck equation (FPE), also known as \emph{Smoluchowski
equation} or \emph{Kolmogorov forward equation}, is one of the major
equations in theoretical physics and applied mathematics. It describes
the time evolution of the probability density function of a particle
in an external force field (e.g., fluctuating forces as in Brownian
motion). The equation can be generalized to other contexts and observables
and has been employed in a broad range of applications, including
physical chemistry, protein synthesis, plasma physics and semiconductor
device simulation. Thus, there is a huge interest in the development
of efficient and robust numerical methods. In the context of finite
volume (FV) methods, the central objective is a robust and accurate
discretization of the (particle or probability) flux implied by the
FPE.

A particularly important discretization scheme for the flux was derived
by Scharfetter and Gummel \cite{SchGum69LSADO} in the context of
the drift-diffusion model for electronic charge carrier transport
in bipolar semiconductor devices \cite{VanRoosbroeck1950}. The typically
exponentially varying carrier densities at p-n junctions lead to unphysical
results (spurious oscillations), if the flux is discretized in a naive
way using standard finite difference schemes \cite{Miller1994}. The
problem was overcome by considering the flux expression as a one-dimensional
boundary value problem along each edge between adjacent mesh nodes.
The resulting Scharfetter--Gummel (SG) scheme provides a robust discretization
of the flux as it asymptotically approaches the numerically stable
discretizations in the drift- (upwind scheme) and diffusion-dominated
(central finite difference scheme) limits. The SG-scheme and its several
generalizations to more complex physical problem settings are nowadays
widely used in semiconductor device simulation \cite{Markowich1986,Farrell2017b}
and have been extensively studied in the literature \cite{Brezzi1989,EyFuGa06FVSNPE,Farrell2017a,Kantner2020}.
The SG-scheme is also known as \emph{exponential fitting scheme} and
was independently discovered by Allan and Southwell \cite{Allan1955}
and Il'in \cite{Ilin1969} in different contexts.

Recently, an alternative flux discretization method, called \emph{square-root
approximation} (SQRA) scheme, has been derived explicitly for high
dimensional problems. The original derivation in \cite{LieFackWeb2013}
aims at applications in molecular dynamics and is based on Markov
state models. However, it can also be obtained from a maximum entropy
path principle \cite{dixit2015inferring} and from discretizing the
Jordan--Kinderlehrer--Otto variational formulation of the FPE \cite{Mielke2013Geodesic}.
In Section \ref{subsec:The-Wasserstein-gradient}, we provide a derivation
of SQRA scheme, which is motivated from the theory of gradient flows.
In contrast to the SG-scheme, the SQRA is very recent and only sparsely
investigated. The only contributions on the convergence seem to be
\cite{Mielke2013Geodesic} in 1D, \cite{DonHeiWebKel2017} (formally,
rectangular meshes) and \cite{heida2018convergences} using G-convergence.

The SG and the SQRA schemes both turn out to be special cases of a
family of discretization schemes based on weighted Stolarsky means,
see Section \ref{subsec:A-family-of-schemes}. This family is very
rich and allows for a general convergence and consistency analysis,
which we carry out in Sections \ref{sec:Comparison-of-discretization}--\ref{sec:Convergence-of-the}.
Interestingly, there seems to be no previous results in the literature.

\subsection{The FPE and the SG and SQRA discretization schemes}

In this work, we consider the stationary Fokker--Planck equation
\begin{equation}
-\nabla\cdot\left(\kappa\nabla u\right)-\nabla\cdot\left(\kappa u\nabla V\right)=f,\label{eq:Fokker-Planck-residual}
\end{equation}
which can be equivalently written as
\[
\mathop{\diver}\mathbf{J}(u,V)=f
\]
using the flux $\mathbf{J}(u,V)=-\kappa\left(\nabla u+u\nabla V\right)$,
where $\kappa>0$ is a (possibly space-dependent) diffusion coefficient
and $V:\Omega\rightarrow\R$ is a given potential.\textcolor{red}{}
The flux $\mathbf{J}$ consists of a diffusive part $\kappa\nabla u$
and a drift part $\kappa u\nabla V$, which compensate for the stationary
density $\pi=\e^{-V}$ (also known as the Boltzmann distribution)
as $\mathbf{J}(\e^{-V},V)=0$. This reflects the principle of detailed
balance in the thermodynamic equilibrium. The right-hand side $f$
describes possible sink or source terms.%

The SG and the SQRA schemes of the Fokker--Planck operator $\diver\,\mathbf{J}(u,V)$
that are considered below are given in the form
\begin{equation}
\left(\cF_{B}^{\cT}u\right)_{i}:=-\sum_{j:\,i\sim j}\frac{m_{ij}}{h_{ij}}\kappa_{ij}\left(B\left(V_{i},V_{j}\right)u_{j}-B\left(V_{j},V_{i}\right)u_{i}\right),\label{eq:FP-discr-Bernoulli}
\end{equation}
where $\sum_{j:\,j\sim i}$ indicates a sum over all cells adjacent
to the $i$-th cell of the mesh, $m_{ij}$ is the mass of the interface
between the $i$-th and $j$-th cell, $h_{ij}$ is the distance between
the corresponding nodes and $\kappa_{ij}$ is the discretized diffusion
coefficient $\kappa$. We are particularly interested in the two cases
\begin{align}
B\left(V_{i},V_{j}\right) & =B_{1}\left(V_{i}-V_{j}\right):=\frac{V_{i}-V_{j}}{\e^{V_{i}-V_{j}}-1}\label{eq:bernoulli-B}\\
\text{or}\quad B\left(V_{i},V_{j}\right) & =B_{2}\left(V_{i}-V_{j}\right):=\e^{-\frac{1}{2}\left(V_{i}-V_{j}\right)}\label{eq:SQRA-B}
\end{align}
with either the Bernoulli function $B_{1}$ (for SG) or with the SQRA-coefficient
$B_{2}$. The schemes are derived under the assumption of constant
flux, diffusion constant and potential gradient along the respective
edges.

In the pure diffusion regime, i.e., for $V_{i}-V_{j}\to0$, the Bernoulli
function provides $B_{1}\left(V_{i}-V_{j}\right)\rightarrow1$, such
that the SG scheme approaches a discrete analogue of the diffusive
part of the continuous flux: $J_{ij}=\kappa_{ij}(u_{i}-u_{j})/h_{ij}$.
In the drift-dominated regime, i.e., for $V_{j}-V_{i}\rightarrow\pm\infty$,
the asymptotics of $B_{1}$ recover the upwind scheme
\begin{equation}
J_{i,j}\to-\kappa_{i,j}\frac{V_{j}-V_{i}}{h_{i,j}}\begin{cases}
u_{j} & \text{if }V_{j}>V_{i}\\
u_{i} & \text{if }V_{j}<V_{i}
\end{cases},\label{eq:limit-B-SG}
\end{equation}
which is a robust discretization of the drift part of the flux, where
the density $u$ is evaluated in the donor cell of the flux. Hence,
the Bernoulli function $B_{1}$ interpolates between the appropriate
discretizations for the drift- and diffusion-dominated limits, which
is why the SG scheme is the preferred FV scheme for Fokker--Planck
type operators. Indeed, the SQRA scheme is consistent with the diffusive
limit, but is less accurate than the SG scheme in the case of strong gradients $\nabla V$.

\subsection{The Stolarsky mean approximation schemes \label{subsec:The-Stolarsky-mean}}

In this work, we investigate the relative $L^{2}$-distance between
the discrete SQRA and SG solutions on the same mesh and the order
of convergence of the SQRA scheme, which was an open problem. It turns
out that both methods are members of a broad family of finite volume
discretizations that stem from the weighted Stolarsky means
\[
S_{\alpha,\beta}\left(x,y\right)=\left(\frac{\beta\left(x^{\alpha}-y^{\alpha}\right)}{\alpha\left(x^{\beta}-y^{\beta}\right)}\right)^{\frac{1}{\alpha-\beta}},
\]
see Section~\ref{subsec:A-family-of-schemes}. We benefit from the
general structure of these schemes and prove order of convergence
on consistent meshes in the sense of the recent work \cite{di2018third}.
We will see that the error naturally splits into the consistency error
for the discretization of the Laplace operator plus an error which
is due to the discretization of the stationary solution $\pi$ and
the Stolarsky mean, see Theorem~\ref{thm:ConsistencyError}.

We will demonstrate below that the Stolarsky discretization schemes
for \eqref{eq:Fokker-Planck-residual} read
\begin{align}
-\sum_{j:\,j\sim i}\frac{m_{ij}}{h_{ij}}\kappa_{ij}S_{ij}\left(\frac{u_{j}}{\pi_{j}}-\frac{u_{i}}{\pi_{i}}\right) & =f_{i},\label{eq:Fokker-Planck-SQRA-residual}
\end{align}
where $\pi_{i}=\e^{-V_{i}}$ , $f_{i}=\int_{\Omega_{i}}f$ is the
integral of $f$ over the $i$-th cell and $S_{ij}=S_{\alpha,\beta}\left(\pi_{i},\pi_{j}\right)$
is a Stolarsky mean of $\pi_{i}$ and $\pi_{j}$. We sometimes refer
to the general form \eqref{eq:Fokker-Planck-SQRA-residual} as discrete
FPE.

The Stolarsky means $S_{\alpha,\beta}$ generalize Hölder means and
other $f$-means (see Table \ref{tab:Choices-for-S}). An interesting aspect of the above representation
is that all these schemes preserve positivity with the discrete linear
operator being an $M$-matrix. Furthermore, with the relative density
$U=u/\pi$ we arrive at
\[
-\sum_{j:\,j\sim i}\frac{m_{ij}}{h_{ij}}\kappa_{ij}S_{ij}\left(U_{j}-U_{i}\right)=f_{i},
\]
which is a discretization of the elliptic equation
\[
-\nabla\cdot\left(\kappa\pi\nabla U\right)=f,
\]
where the discrete Fokker--Planck operator becomes a purely diffusive
second order operator in $U$. Furthermore, if $\kappa$ is a symmetric
strictly positive definite uniformly elliptic matrix, this operator
is also symmetric strictly positive definite and uniformly elliptic.
In the latter setting, we can thus rule out the occurrence of spurious
oscillations in our discretization.%

Although we treat the Stolarsky means as an explicit example, note
that the main theorems also hold for other smooth means.

\subsection{Major contributions of this work}

Since we look at the FV discretization of the FPE from a very broad
point of view, we summarize our major findings.
\begin{itemize}
\item We provide a derivation of the general Stolarsky mean FV discretization
in Section~\ref{subsec:A-family-of-schemes}.
\item We discuss the gradient structure of the discretization schemes in
view of the natural gradient structure of the FPE in Section~\ref{subsec:The-Wasserstein-gradient}.
\item We provide order of convergence of the schemes as the fineness of
the discretization tends to zero. In particular we show
\begin{itemize}
\item that the order of convergence is mainly determined by two independent
parts: the consistency (Def.~\ref{def:consistency}) of the mesh
(Section~\ref{subsec:Error-Analysis-in-U}) and an error due to the
discretization of $\pi$ along the edge by the means $S_{\ast}(\pi_{i},\pi_{j})$.
\item that the order of convergence strongly depends on the constant $\alpha+\beta$,
where $\alpha$ and $\beta$ are the Stolarsky coefficients in $S_{ij}=S_{\alpha,\beta}\left(\pi_{i},\pi_{j}\right)$
(Corollary~\ref{cor:ComparisonDifferenceForDifferentStolarskyMean}).
\item that the SG coefficients are to be preferred in regions of strong
gradient $\nabla V$ (Section~\ref{subsec:Error-Analysis-in-u}).
\end{itemize}
\end{itemize}

\subsection{Outlook}

The results of this work suggest to search for ``optimal'' parameters
$\alpha$ and $\beta$ in the choice of the Stolarsky mean in order
to reduce the error of the approximation as much as possible. However,
from an analytical point of view, the quest for such optimal $\alpha$
and $\beta$ is quite challenging. Moreover, since the optimal choice
might vary locally, depending on the local properties of the potential
$V$, we suggest to implement a learning algorithm that provides suitable
parameters $\alpha$ and $\beta$ depending on the local structure
of $V$ and the mesh.

\subsection{Outline of this work}

After some preliminaries regarding notation and a priori estimates
in Section~\ref{sec:Preliminaries-and-Notation}, we will recall
the classical derivation of the SG scheme in Section~\ref{subsec:Motivation-of-SG}
and discuss its formal relation to SQRA. We will then provide a derivation
of SQRA from physical principles in Section~\ref{subsec:The-Wasserstein-gradient},
based on the Jordan--Kinderlehrer--Otto \cite{jordan1998variational}
formulation of the FPE. In Section \ref{subsec:A-family-of-schemes},
we show that SG and SQRA are elements of a huge family of discretization
schemes~\eqref{eq:Fokker-Planck-SQRA-residual}.

Section \ref{sec:Convergence-of-the} provides the error analysis
and estimates for the consistency and the order of convergence. We
distinguish the cases of small and large gradients and have a particular
look at cubic meshes.

Finally, we show hat the optimal choice of $S_{\ast}$ depends on
$V$ and $f$ but is not unique. If $S_{\alpha,\beta}$ denotes one
of the Stolarsky means, we will prove in Section~\ref{sec:Comparison-of-discretization}
that the Stolarsky means satisfying $\alpha+\beta=\mathrm{const}$
show similar quantitative convergence behavior as suggested in Corollary
\ref{cor:ComparisonDifferenceForDifferentStolarskyMean}. Finally,
this result is illustrated in Section~\ref{sec:Interpretation-and-simulation}
by numerical simulations.

\section{\label{sec:Preliminaries-and-Notation}Preliminaries and notation}

We collect some concepts and notation, which will frequently be used
in this work.

\subsection{The Mesh}

For a subset $A\subset\R^{d}$, $\overline{A}$ is the closure of
$A$.
\begin{defn}
\label{def:admissible-grid}Let $\Omega\subset\Rd$ be a polygonal
domain. A finite volume mesh of $\Omega$ is a triangulation $\mathcal{T}=(\mathcal{V},\mathcal{E},\mathcal{P})$
consisting of a family of control volumes $\mathcal{V}:=\left\{ \Omega_{i},i=1,\dots,N\right\} $
which are convex polytope cells, a family of $\left(d-1\right)$-dimensional
interfaces
\begin{align*}
\mathcal{E} & :=\mathcal{E}_{\Omega}\cup\mathcal{E_{\partial}}\\
\mathcal{E}_{\Omega} & :=\left\{ \sigma_{ij}\subset\Rd\,:\;\sigma_{ij}=\partial\Omega_{i}\cap\partial\Omega_{j}\right\} \\
\mathcal{E}_{\partial} & :=\left\{ \sigma\subset\Rd\,:\;\sigma=\partial\Omega_{i}\cap\partial\Omega\,\text{\ is flat}\right\} 
\end{align*}
and points $\mathcal{P}=\{x_{i},i=1,\dots,N\}$ with $x_{i}\in\overline{\Omega_{i}}$
satisfying
\begin{enumerate}
\item [(i)]$\bigcup_{i}\overline{\Omega}_{i}=\overline{\Omega}$
\item [(ii)]For every $i$ there exists $\mathcal{E}_{i}\subset\mathcal{E}$
such that $\overline{\Omega}_{i}\backslash\Omega_{i}=\bigcup_{\sigma\in\mathcal{E}_{i}}\sigma$.
Furthermore, $\mathcal{E}=\bigcup_{i}\mathcal{E}_{i}$.
\item [(iii)]For every $i,j$ either $\overline{\Omega}_{i}\cap\overline{\Omega}_{j}=\emptyset$
or $\overline{\Omega}_{i}\cap\overline{\Omega}_{j}=\overline{\sigma}$
for $\sigma\in\mathcal{E}_{i}\cap\mathcal{E}_{j}$ which will be denoted
$\sigma_{ij}$.
\end{enumerate}
The mesh is called $h$-consistent if
\begin{enumerate}
\item [(iv)] The Family $\left(x_{i}\right)_{i=1\dots N}$ is such that
$x_{i}\not=x_{j}$ if $i\not=j$ and the straight line $D_{ij}$ going
through $x_{i}$ and $x_{j}$ is orthogonal to $\sigma_{ij}$.
\end{enumerate}
and admissible if
\begin{enumerate}
\item [(v)] For any boundary interface $\sigma\in\mathcal{E}_{\partial}\cap\mathcal{E}_{i}$
it holds $x_{i}\not\in\sigma$ and for $D_{i,\sigma}$ the line through
$x_{i}$ orthogonal to $\sigma$ it holds that $D_{i,\sigma}\cap\sigma\not=\emptyset$
and let $y_{\sigma}:=D_{i,\sigma}\cap\sigma$.
\end{enumerate}
\end{defn}

Property (iv) is assumed in \cite{gallouet2000error} in order to
prove a strong form of consistency in the sense of Definition \ref{def:consistent}
below. It is satisfied for example for Voronoi discretizations.

We write $m_{i}$ for the volume of $\Omega_{i}$ and for $\sigma\in\mathcal{E}$
we denote $m_{\sigma}$ its $\left(d-1\right)$-dimensional mass.
In case $\sigma_{ij}\in\mathcal{E}_{i}\cap\mathcal{E}_{j}$ we write
$m_{ij}:=m_{\sigma_{ij}}$. For the sake of simplicity, we consider
$\tilde{\cP}:=\left(x_{i}\right)_{i=1,\dots,N}$ and $\cP:=\tilde{\cP}\cup\left\{ y_{\sigma}\,:\;\sigma\in\mathcal{E}_{\partial},\,\text{according to (v)}\right\} $.
We extend the enumeration of $\tilde{\cP}$ to $\cP=\left(x_{j}\right)_{j=1,\dots,\tilde{N}}$
and write $i\sim j$ if $x_{i},x_{j}\in\tilde{\cP}$ with $\mathcal{E}_{i}\cap\mathcal{E}_{j}\not=\emptyset$.
Similarly, if $x_{i}\in\tilde{\cP}$ and $x_{j}=y_{\sigma}$ for $\sigma\in\mathcal{E}_{i}$
we write $\sigma_{ij}:=\sigma$ and $i\sim j$. Finally, we write
$h_{ij}=\left|x_{i}-x_{j}\right|$.

We further call
\[
\mathcal{P}^{*}:=\left\{ u:\,\cP\to\R\right\} \,,\quad\tilde{\mathcal{P}}^{*}:=\left\{ u:\,\tilde{\cP}\to\R\right\} \,,\quad\text{and}\quad\cE^{*}:=\left\{ w:\,\mathcal{E}\to\R\right\} 
\]
the discrete functions from $\cP$ resp. $\tilde{\cP}$ resp. $\cE$
to $\R$. For $w\in\cE^{*}$ we write $w_{ij}:=w(\sigma)$ if $\sigma_{ij}=\sigma$.
Then for fixed $i$ the expression
\[
\sum_{j:\,i\sim j}w_{ij}:=\sum_{\sigma_{ij}\in\cE_{i}}w_{ij}
\]
is the sum over all $w_{ij}$ such that $\cE_{i}\cap\cE_{j}\not=\emptyset$
and
\[
\sum_{j\sim i}w_{ij}:=\sum_{\sigma_{ij}\in\cE}w_{ij}:=\sum_{\sigma\in\cE}w(\sigma)
\]
is the sum over all edges.

Moreover, we define the diameter of a triangulation $\cT$ as
\[
\diam\cT=\sup_{i\sim j}{|x_{i}-x_{j}|}.
\]

\renewcommand{\arraystretch}{1.2}

\begin{table}[t]
\begin{tabular}{>{\centering}p{1.7cm} p{5cm} p{0.49cm} >{\centering}p{1.7cm} p{5cm}}
\toprule
symbol & meaning &  & symbol & meaning\tabularnewline
\cline{1-2} \cline{2-2} \cline{4-5} \cline{5-5} 
$u$ & density &  & $m_{i}$ & $\mathrm{vol}(\Omega_{i})$\tabularnewline
$V$ & real potential on $\Omega\subset\R^{d}$ &  & $h_{i}$ & $\mathrm{diam}(\Omega_{i})$\tabularnewline
$\kappa$ & diffusion coefficient &  & $\sigma_{ij}$ & $\partial\Omega_{i}\cap\partial\Omega_{j}$\tabularnewline
$\pi$ & stat. measure $\e^{-V(x)}$ on $\Omega$ &  & $m_{ij}$ & area of $\sigma_{ij}$\tabularnewline
$U$ & $u/\pi$ &  & $\mathbf{h}_{ij}$ & $x_{i}-x_{j}$\tabularnewline
$u_{i}$ & $u(x_{i})$ &  & $h_{ij}$ & $|\mathbf{h}_{ij}|$\tabularnewline
$\pi_{i}$ & stat. measure $\e^{-V(x_{i})}$ on $\Omega_{i}$ &  & $d_{i,ij}$ & $\mathrm{dist}\left(x_{i},\sigma_{ij}\right)$\tabularnewline
$\bar{f}_{i}$ & $\frac{1}{|\Omega_{i}|}\int_{\Omega_{i}}f\d x$ &  & $\diam\cT$ & diameter, i.e. $\sup_{i\sim j}|x_{i}-x_{j}|$\tabularnewline
$f_{i}$ & $m_{i}\bar{f}_{i}$ &  & $V^{*},V_{*}$ & $-\infty<V_{*}\leq V\leq V^{\ast}<\infty$\tabularnewline
$\kappa_{ij}$ & $\frac{h_{ij}\bar{\kappa}_{i}\bar{\kappa}_{j}}{\bar{\kappa}_{i}d_{i,ij}+\bar{\kappa}_{j}d_{j,ij}}$ &  & $\kappa^{*},\kappa_{*}$ & $0<\kappa_{*}\leq\kappa\leq\kappa^{\ast}<\infty$\tabularnewline
\bottomrule
\end{tabular}
\caption{\label{tab:Notation}Commonly used notations.}
\end{table}

The identity
\begin{equation}
\sum_{i}\sum_{j:j\sim i}A_{ij}=\sum_{j\sim i}\left(A_{ij}+A_{ji}\right)\label{eq:summen-formel}
\end{equation}
will frequently be used throughout this paper, where we often encounter
the case $A_{ij}=\alpha_{ij}U_{i}$ with $\alpha_{ij}=-\alpha_{ji}$:
\begin{equation}
\sum_{i}\sum_{j:j\sim i}\alpha_{ij}U_{i}=\sum_{j\sim i}\left(\alpha_{ij}U_{i}+\alpha_{ji}U_{j}\right)=\sum_{j\sim i}\alpha_{ij}\left(U_{i}-U_{j}\right)\,.\label{eq:general-disc-partial-int}
\end{equation}
Formula \eqref{eq:summen-formel} in particular allows for a discrete
integration by parts:
\begin{equation}
\sum_{i}\sum_{j:j\sim i}\left(U_{j}-U_{i}\right)U_{i}=\sum_{j\sim i}\left(\left(U_{j}-U_{i}\right)U_{i}+\left(U_{i}-U_{j}\right)U_{j}\right)=-\sum_{j\sim i}\left(U_{j}-U_{i}\right)^{2}\,.\label{eq:discrete-P-I}
\end{equation}
On a given mesh $\mathcal{T}=(\mathcal{V},\mathcal{E},\mathcal{P})$,
we consider the linear discrete operator $\cL_{\kappa}^{\cT}:\,\cP^{\ast}\to\cP^{\ast}$,
which is defined by a family of non-negative weights $\kappa:\,\cE\to\R$
and acts on functions $u\in\cP^{\ast}$ via
\begin{equation}
\forall x_{i}\in\cP:\;\left(\cL_{\kappa}^{\cT}u\right)_{i}:=\sum_{i\sim j}\kappa_{ij}\frac{m_{ij}}{h_{ij}}\left(u_{j}-u_{i}\right)\,.\label{eq:general-linear-Op}
\end{equation}
While \eqref{eq:general-linear-Op} is very general, it is shown in
\cite{gallouet2000error}, Lemma 3.3, that the property (iv) of Definition~\ref{def:admissible-grid}
comes up with some special consistency properties for the choice of
\begin{equation}
\kappa_{ij}:=\frac{\bar{\kappa}_{i}\bar{\kappa}_{j}}{\bar{\kappa}_{i}\frac{d_{i,ij}}{h_{ij}}+\bar{\kappa}_{j}\frac{d_{j,ij}}{h_{ij}}}\,,\label{eq:kappa-ij}
\end{equation}
where $d_{i,ij}$ and $d_{j,ij}$ are the distances between $\sigma_{ij}$
and $x_{i}$ and $x_{j}$ respectively and averaged diffusion coefficient
is defined by $\overline{\kappa}_{i}=\tfrac{1}{m_{i}}\int_{\Omega_{i}}\!\kappa(x)\d x$.
\begin{lem}[A consistency lemma, \cite{gallouet2000error}]
\label{lem:Gal-Herb-Vign}Let the $\mathcal{T}=(\mathcal{V},\mathcal{E},\mathcal{P})$
satisfy Definition \ref{def:admissible-grid} (i)--(v) and let $d\in\left\{ 2,3\right\} $
and let $h_{ij}$ be uniformly bounded from above and from below.
Then for every $u\in H^{2}\left(\Omega\right)$ it holds
\[
\left|\int_{\sigma_{ij}}\kappa\nabla u\cdot\bnu_{ij}-\kappa_{ij}\frac{m_{ij}}{h_{ij}}\left(u\left(x_{j}\right)-u\left(x_{i}\right)\right)\right|\leq Cm_{ij}^{\frac{1}{2}}h_{ij}^{\frac{1}{2}}\left\Vert u\right\Vert _{H^{2}(\Omega_{i}\cup\Omega_{j})}\,.
\]
\end{lem}

Lemma \ref{lem:Gal-Herb-Vign} was one of the motivations to provide
a more general and powerful concept of consistency in \cite{di2018third},
as we will discuss in Section \ref{subsec:Consistency-and-Inf-sup}

\subsection{\label{subsec:Existence-and-a}Existence and a priori estimates}

From the standard theory of elliptic systems (\cite{Evans1998} Chapter
6), we have the following theorem.
\begin{thm}
Let $\Omega$ be as above and $f\in\mathrm{L}^{2}(\Omega)$, $\kappa\in C^{1}\left(\overline{\Omega}:\R^{d\times d}\right)$
such that $\kappa$ is uniformly bounded, symmetric and elliptic and
$V\in\mathrm{C}^{2}(\overline{\Omega})$. Then there is a unique $u\in\mathrm{H}^{2}(\Omega)\cap\mathrm{H}_{0}^{1}(\Omega)$
solving $-\nabla\cdot\left(\kappa\nabla u\right)-\nabla\cdot\left(\kappa u\nabla V\right)=f$
in the weak sense.
\end{thm}

In what follows, we frequently use the following transformations in
\eqref{eq:Fokker-Planck-residual} and \eqref{eq:Fokker-Planck-SQRA-residual}:
we define the relative densities $U=u/\pi$ and $U_{i}^{\cT}=u_{i}^{\cT}/\pi_{i}$
to find
\begin{align}
-\nabla\cdot\left(\kappa\pi\nabla U\right) & =f\,,\label{eq:SQRA-conv-1}\\
\forall i:\qquad-\sum_{j:\,j\sim i}\frac{m_{ij}}{h_{ij}}\kappa_{ij}S(\pi_{i},\pi_{j})\left(U_{j}^{\cT}-U_{i}^{\cT}\right) & =f_{i}\,.\label{eq:SQRA-conv-2}
\end{align}

If $\kappa$ and $\pi$ are non-degenerate (in the sense of $\pi>c>0$
and $\xi\cdot\kappa\xi>c\left|\xi\right|^{2}$), the left hand side
of \eqref{eq:SQRA-conv-2} defines a strongly elliptic operator on
the finite volume space $L^{2}(\cP)$ and hence there exists a unique
solution $U^{\cT}$. Concerning the right hand side, using \eqref{eq:SQRA-conv-2}
as the discretization of \eqref{eq:SQRA-conv-1}, one natural choice
for $f_{i}$ is $f_{i}=m_{i}\bar{f}_{i}$, where $\bar{f}_{i}=m_{i}^{-1}\int_{\Omega_{i}}f$.
We immediately see that the Boltzmann distribution $\pi_{i}=\exp\left(-V(x_{i})\right)=\exp\left(-V_{i}\right)$
is the stationary solution, i.e., $u_{i}=\pi_{i}$ solves \eqref{eq:SQRA-conv-2}
for $f=0$.

Having shown the existence of solutions to \eqref{eq:SQRA-conv-1}
and \eqref{eq:SQRA-conv-2}, we recall the derivation of some natural
a priori estimates for both the continuous Fokker--Planck equation
and the discretization.

\paragraph*{Continuous FPE}

Let $u$, resp. $U=u/\pi$, be a solution of the stationary Fokker--Planck
equation \eqref{eq:SQRA-conv-1} with Dirichlet boundary conditions.
Testing with $U$, we get (assuming homogeneous Dirichlet boundary
conditions and exploiting thus the Poincaré inequality), that
\begin{gather}
\int_{\Omega}\kappa\pi|\nabla U|^{2}=\int_{\Omega}Uf\ \ \leq C\left(\int_{\Omega}f^{2}\right)^{\frac{1}{2}}\left(\int_{\Omega}\kappa\pi|\nabla U|^{2}\right)^{\frac{1}{2}}\nonumber \\
\Rightarrow\int_{\Omega}\tfrac{1}{\kappa\pi}|\kappa\pi\nabla U|^{2}\leq C\int_{\Omega}f^{2}\,.\label{eq:AprioriGradientU-cont}
\end{gather}
Furthermore, the standard theory of elliptic equations (e.g., \cite{Evans1998})
yields that $\norm U_{H^{2}(\Omega)}\leq C\norm f_{L^{2}}$, where
$C$ depends on the $C^{1}$-norm of $\kappa\pi$ and the Poincaré-constant.

\paragraph*{Discrete FPE}

Let $U_{i}^{\mathcal{T}}$ be a solution of \eqref{eq:SQRA-conv-2}
with $f_{i}=m_{i}\bar{f}_{i}=\int_{\Omega_{i}}f\d x$ (as specified
in the Tab.~\ref{tab:Notation}), i.e.,
\[
\forall i:\qquad-\sum_{j:j\sim i}\frac{m_{ij}}{h_{ij}}\kappa_{ij}S_{ij}\left(U_{j}^{\mathcal{T}}-U_{i}^{\mathcal{T}}\right){\color{black}=}m_{i}\bar{f}_{i}\,.
\]

Then, multiplying with $U_{i}^{\mathcal{T}}$ we get
\[
-\sum_{j:j\sim i}\frac{m_{ij}}{h_{ij}}\kappa_{ij}S_{ij}\left(U_{j}^{\mathcal{T}}-U_{i}^{\mathcal{T}}\right)U_{i}^{\mathcal{T}}=m_{i}\bar{f}_{i}U_{i}^{\mathcal{T}}.
\]

Summing over all $x_{i}\in\cP$ and using \eqref{eq:discrete-P-I},
we conclude with help of the discrete Poincaré inequality (see Theorem
\ref{thm:discrete-PI} below)
\begin{align*}
\sum_{j\sim i}\frac{m_{ij}}{h_{ij}}\kappa_{ij}S_{ij}\left(U_{j}^{\mathcal{T}}-U_{i}^{\mathcal{T}}\right)^{2}=\sum_{i}m_{i}\bar{f}_{i}U_{i}^{\mathcal{T}} & \leq\sum_{i}\left((U_{i}^{\cT})^{2}m_{i}+\tfrac{1}{\pi_{i}}\bar{f}_{i}^{2}m_{i}\right)\\
\Rightarrow\sum_{j\sim i}\frac{m_{ij}}{h_{ij}}\kappa_{ij}S_{ij}\left(U_{j}^{\mathcal{T}}-U_{i}^{\mathcal{T}}\right)^{2} & \leq C\sum_{i}m_{i}\bar{f}_{i}^{2}.
\end{align*}

Additionally, one gets
\begin{equation}
\sum_{j\sim i}\frac{m_{ij}}{h_{ij}}\kappa_{ij}\frac{1}{S_{ij}\kappa_{ij}^{2}}\left(\kappa_{ij}S_{ij}(U_{j}^{\mathcal{T}}-U_{i}^{\mathcal{T}})\right)^{2}\leq C\sum_{i}\bar{f}_{i}^{2}m_{i}.\label{eq:AprioriGradientU-disc}
\end{equation}

\subsection{\label{subsec:Fluxes-and--spaces}Fluxes and $L^{2}$-spaces}

In order to derive and formulate variational consistence errors for
the discrete FPE \textbf{\eqref{eq:SQRA-conv-2}}, we introduce the
discrete fluxes
\begin{equation}
\begin{aligned}J_{ij}^{S}U^{\cT} & :=-\frac{\kappa_{ij}}{h_{ij}}S_{ij}\left(U_{j}^{\cT}-U_{i}^{\cT}\right)\,,\\
\overline{J}_{ij}U & :=-\frac{1}{m_{ij}}\int_{\sigma_{ij}}\kappa\pi\nabla U\cdot\bnu_{ij}\,.
\end{aligned}
\label{eq:fluxes}
\end{equation}
In particular, if $S_{ij}=\sqrt{\pi_{i}\pi_{j}}$ we get the flux
of the SQRA $J_{ij}^{\text{SQRA}}U^{\cT}:=-\frac{\kappa_{ij}}{h_{ij}}\sqrt{\pi_{i}\pi_{j}}\left(U_{j}^{\cT}-U_{i}^{\cT}\right)$.
Note that $\overline{J}_{ij}U$ is the spatial average of $\mathbf{J}(U)=-\kappa\pi\nabla U$
on $\sigma_{ij}$. The quantity $J_{ij}^{S}U^{\cT}$ can indeed be
considered as a flux in the sense that it will be shown to approximate
$\overline{J}_{ij}$, $S_{ij}$ is a discrete approximation of $\pi|_{\sigma_{ij}}$
, $\kappa_{ij}$ is a discrete approximation of $\kappa|_{\sigma_{ij}}$
and $\frac{1}{h_{ij}}\left(U_{j}^{\cT}-U_{i}^{\cT}\right)$ is a discrete
version of $\nabla U$.

While former approaches focus on the rate of convergence of $\frac{1}{h_{ij}}\left(u_{j}^{\cT}-u_{i}^{\cT}\right)\to\nabla u$,
we additionally follow the approach of \cite{di2018third} applied
to $U$ and are interested in the rate of convergence of $J_{ij}^{S}U^{\cT}\to\mathbf{J}(U)$,
which is an indirect approach to the original problem as this rate
of convergence is directly related to $\frac{1}{h_{ij}}\left(U_{j}^{\cT}-U_{i}^{\cT}\right)\to\nabla U$.

In view of the natural norms for the variational consistency (see
\eqref{eq:norm-H-T-kappa}--\eqref{eq:norm-H-T-kappa-}), we introduce
the following
\begin{align}
\forall U\in L^{2}(\Omega): &  & \|U\|_{L^{2}(\Omega)}^{2} & :=\int_{\Omega}U^{2}\d x & \|U\|_{L_{\pi}^{2}(\Omega)}^{2} & :=\int_{\Omega}\tfrac{1}{\pi}U^{2}\d x\nonumber \\
\forall U\in\mathcal{P}^{*}: &  & \|U\|_{L^{2}(\mathcal{P})}^{2} & :=\sum_{i\in\mathcal{P}}m_{i}U_{i}^{2} & \|U\|_{L_{\pi}^{2}(\mathcal{P})}^{2} & :=\sum_{i\in\mathcal{P}}m_{i}\tfrac{1}{\pi_{i}}U_{i}^{2}\label{eq:def-norms}\\
\forall J\in\mathcal{E}^{*}: &  & \|J\|_{L^{2}(\mathcal{E})}^{2} & :=\sum_{i\sim j}m_{ij}h_{ij}J_{ij}^{2} & \|J\|_{L_{S}^{2}(\mathcal{E})}^{2} & :=\sum_{i\sim j}m_{ij}h_{ij}\frac{1}{S_{ij}}J_{ij}^{2}\nonumber 
\end{align}

Let us introduce the discrete flux $J^{S}U^{\cT}\in\mathcal{E}^{*}$
via $J^{S}U^{\cT}(\sigma_{ij}):=J_{ij}^{S}U^{\cT}$ and similarly
also $\frac{1}{\kappa}J^{S}U^{\cT}\in\mathcal{E}^{*}$ via $J^{S}U^{\cT}(\sigma_{ij}):=\frac{1}{\kappa_{ij}}J_{ij}^{S}U^{\cT}$.
With all the above notations, our a priori estimates \eqref{eq:AprioriGradientU-cont}
and \eqref{eq:AprioriGradientU-disc} now read
\begin{align*}
\norm{\frac{1}{\sqrt{\kappa}}\mathbf{J}(U)}_{L_{\pi}^{2}(\Omega)}^{2} & \leq C\|f\|_{L_{\pi}^{2}(\Omega)}^{2}\\
\norm{\frac{1}{\sqrt{\kappa}}J^{S}U^{\cT}}_{L_{S}^{2}(\cE)}^{2} & \leq C\|\bar{f}\|_{L_{\pi}^{2}(\mathcal{P})}^{2}\,.
\end{align*}

Assuming that the diffusion coefficient is bounded, i.e. $\kappa^{*}\geq\kappa\geq\kappa_{*}$,
we further get
\begin{align*}
\tfrac{1}{\kappa_{*}}\|\mathbf{J}(U)\|_{L_{\pi}^{2}(\Omega)}^{2} & \leq C\|f\|_{L_{\pi}^{2}(\Omega)}^{2}\\
\frac{1}{\kappa^{*}}\|J^{S}U^{\cT}\|_{L_{S}^{2}(\mathcal{E})}^{2} & \leq C\|\bar{f}\|_{L_{\pi}^{2}(\mathcal{P})}^{2}\,.
\end{align*}

\begin{rem}[Naturalness of norms]
 Let us discuss why these norms are natural to consider. The left
norms in \eqref{eq:def-norms} can be interpreted as the Euclidean
$L^{2}$-norms on $\Omega$, $\cP$ and $\cE$, while the right norms
are the natural norms for the study of the Fokker--Planck equation
as they are weighted with the inverse of the Boltzmann distribution
$\pi$, resp. $\pi_{i}$. Note that assuming $V$ is bounded from
above and below, the $L^{2}$-norms $\|\cdot\|_{L_{\pi}^{2}(\Omega)}$
and $\|\cdot\|_{L^{2}(\Omega)}$ are equivalent and the same holds
true for the two norms in the discrete setting.

Given a discretization $\cT$, the linear map
\[
C_{c}\left(\Rd\right)\to\R\,,\qquad f\mapsto\sum_{i\in\mathcal{P}}m_{i}f(x_{i})
\]
defines an integral on $\Omega$ w.r.t. a discrete measure $\mu_{\cT}$
having the property that $\mu_{\cT}\to\cL^{d}$ vaguely, where $\cL^{d}$
is the $d$-dimensional Lebesgue measure. In particular $\mu_{\cT}(A)\to\cL^{d}\left(A\right)$
for every bounded measurable set with $\cL^{d}\left(\partial A\right)=0$.
The norm $\|U\|_{L^{2}(\mathcal{P})}^{2}$ is simply the $L^{2}$-norm
based on the measure $\mu_{\cT}$.

Similar considerations work also for the norm on ${\mathcal E}^{*}$.
The norm $\|\cdot\|_{L^{2}(\mathcal{E})}^{2}$ is given via a measure
$\tilde{\mu}_{\cT}$ having the property
\[
\tilde{\mu}_{\cT}:\,C_{c}\left(\Rd\right)\to\R\,,\qquad f\mapsto\sum_{i\sim j}m_{ij}h_{ij}f(x_{ij})\,,
\]
with the property that $\tilde{\mu}_{\cT}\to d\cdot\cL^{d}$ vaguely:
every Voronoi cell $\Omega_{i}$ consists of disjoint cones with mass
$\frac{1}{d}m_{ij}h_{ij}$, where one has to account for all cones
with $j\sim i$. In particular, we obtain $\tilde{\mu}_{\cT}(A)\approx d\cdot\cL(A)$
for Lipschitz domains -- an estimate which then becomes precise in
the limit. Without going into details, let us mention that heuristically
the prefactor $d$ balances the fact that $J_{ij}\approx\frac{\left(x_{i}-x_{j}\right)}{\left|x_{i}-x_{j}\right|}\cdot\nabla U$
which yields for functions $U\in C_{c}^{1}\left(\Rd\right)$: 
\[
\sum_{i\sim j}m_{ij}h_{ij}\left|\frac{\left(x_{i}-x_{j}\right)}{\left|x_{i}-x_{j}\right|}\cdot\nabla U\right|^{2}\to\int_{\Rd}\left|\nabla U\right|^{2}\,.
\]
For the particular case of a rectangular mesh, this is straight forward
to verify.
\end{rem}

\subsection{\label{subsec:Poincar=0000E9-inequalities-and}Poincaré inequalities}

In order to derive the a priori estimates in Section \ref{subsec:Existence-and-a}
we need to exploit (discrete) Poincaré inequalities to estimate $\norm u_{L^{2}\left(\Omega\right)}$
by $\norm{\nabla u}_{L^{2}\left(\Omega\right)}$ or $\norm{u^{\cT}}_{L^{2}\left(\cP\right)}$
by $\norm{Du^{\cT}}_{L^{2}(\cE)}$, where $\left(Du^{\cT}\right)_{ij}=U_{j}-U_{i}$.
In particular, we use the following theorem.
\begin{thm}
\label{thm:discrete-PI}Given a mesh $\cT=\left(\mathcal{V},\cE,\cP\right)$
let $h_{\mathrm{inf}}:=\inf\left\{ \left|x-y\right|\,:\;\left(x,y\right)\in\cP^{2}\right\} >0$
and $h_{\mathrm{sup}}:=\sup\left\{ \left|x-y\right|\,:\;\left(x,y\right)\in\cP^{2}\right\} >0$
correspondingly. Then for every $u\in L^{2}\left(\cP\right)$ and
for every $\boldsymbol{\eta}\in\R^{d}$ it holds
\begin{equation}
\int_{\Omega}\left|\sum_{i}u_{i}\chi_{\Omega_{i}}(x)-\sum_{i}u_{i}\chi_{\Omega_{i}}(x+\boldsymbol{\eta})\right|^{2}dx\leq\left|\boldsymbol{\eta}\right|\left(\diam\Omega\frac{h_{\mathrm{sup}}}{h_{\mathrm{inf}}}\sum_{i\sim j}\frac{m_{ij}}{h_{ij}}\left(u_{j}-u_{i}\right)^{2}\right)\,,\label{eq:Poincare-L-2-disc}
\end{equation}
and particularly
\[
\norm u_{L^{2}\left(\cP\right)}^{2}\leq\left(\diam\Omega\right)^{2}\frac{h_{\mathrm{sup}}}{h_{\mathrm{inf}}}\sum_{i\sim j}m_{ij}\left(u_{j}-u_{i}\right)^{2}\,.
\]
\end{thm}

\begin{proof}
This follows from Lemma \ref{lem:Poincare} with $C_{\#}\leq\frac{\diam\Omega}{h_{0}}$
and the choice $\left|\boldsymbol{\eta}\right|>\diam\Omega$.
\end{proof}

\subsection{\label{subsec:Consistency-and-Inf-sup}Consistency and inf-sup stability}

Results such as Lemma \ref{lem:Gal-Herb-Vign} motivated the authors
of the recent paper \cite{di2018third} to define the concepts of
consistency and inf-sup stability as discussed in the following. For
readability, we will restrict the general framework of \cite{di2018third}
to cell-centered finite volume schemes and refer to general concepts
only as far as needed.
\begin{defn}[inf-sup stability]
 A bilinear form $a_{\cT}$ on $L^{2}\left(\cP\right)$ for a given
mesh $\mathcal{T}=(\mathcal{V},\mathcal{E},\mathcal{P})$ is called
inf-sup stable with respect to a norm $\norm{\cdot}_{H_{\cT}}$ on
a subspace of $H_{\cT}\subset L^{2}\left(\cP\right)$ if there exists
$\gamma>0$ such that
\[
\forall u\in H_{\cT}\,:\quad\gamma\norm u_{H_{\cT}}\leq\sup_{v\in H_{\cT}}\frac{a_{\cT}\left(u,v\right)}{\norm v_{H_{\cT}}}\,.
\]
\end{defn}

Usually, and particularly in our setting, $a_{\cT}$ is the discretization
of a continuous bilinear form, say $a\left(u,v\right)=\int_{\Omega}\nabla u\cdot\left(\kappa\nabla v\right)$.
We are interested in the problem
\begin{equation}
\forall v\in H_{0}^{1}\left(\Omega\right)\::\quad a\left(u,v\right)=l\left(v\right)\,,\label{eq:abstract linear}
\end{equation}
where $l:\,H_{0}^{1}\left(\Omega\right)\to\R$ is a continuous linear
map, and in the convergence of the solutions of the discrete problems
\begin{equation}
\forall v\in L^{2}\left(\cT\right)\::\quad a_{\cT}\left(u_{\cT},v\right)=l_{\cT}\left(v\right)\,.\label{eq:abstract linear-disc}
\end{equation}

\begin{defn}[Consistency]
\label{def:consistency} Let $B\subset H_{0}^{1}\left(\Omega\right)$
be a continuously embedded Banach subspace and for given $\mathcal{T}=(\mathcal{V},\mathcal{E},\mathcal{P})$
consider continuous linear operators $\cR_{\cT}:\,B\to L^{2}\left(\cP\right)$
with uniform bound. Let $u$ be the solution to the linear equation
\eqref{eq:abstract linear} and let $l_{\cT}:\,L^{2}\left(\cP\right)\to\R$
be a family of linear functionals. The variational consistency error
is the linear form $\mathfrak{E}_{\cT}\left(u;\cdot\,\right):\,L^{2}\left(\cP\right)\to\R$
where
\[
\forall u\in B\,:\qquad\mathfrak{E}_{\cT}\left(u;\cdot\,\right):=l_{\cT}\left(\cdot\right)-a_{\cT}\left({\mathcal R}_{\mathcal{T}}u,\cdot\,\right)\,.
\]
Let now a family $\left(\cT,a_{\cT},l_{\cT}\right)$ with $\diam\cT\to0$
be given and consider the corresponding family of linear discrete
problems \eqref{eq:abstract linear-disc}. We say that consistency
holds if
\[
\norm{\mathfrak{E}_{\cT}\left(u;\cdot\,\right)}_{H_{\cT}^{\ast}}\to0\quad\text{as}\quad\diam\cT\to0\,.
\]
\end{defn}

\begin{rem}
A typical situation is the case $d\leq3$, where $H^{2}\left(\Omega\right)\cap H_{0}^{1}\left(\Omega\right)\hookrightarrow C_{0}\left(\Omega\right)$
continuously. We then might set $B=H^{2}\left(\Omega\right)\cap H_{0}^{1}\left(\Omega\right)$
and $\left(\cR_{\cT}u\right)_{i}:=u\left(x_{i}\right)$.
\end{rem}

Consistency measures the rate at which $\cR_{\cT}u-u_{\cT}\to0$ and
particularly provides a positive answer to the question whether the
numerical scheme converges, at least if the solution of \eqref{eq:abstract linear}
lies in $B$. This is formulated in Theorem 10 of \cite{di2018third}.
\begin{thm}[Theorem 10, \cite{di2018third}]
Using the above notation, it holds
\begin{equation}
\norm{u_{\cT}-\cR_{\cT}u}_{H_{\cT}}\leq\gamma^{-1}\norm{\fE_{\cT}\left(u;\cdot\,\right)}_{H_{\cT}^{\ast}}\label{eq:general-conv-order}
\end{equation}
\end{thm}

In our setting, $\|\cdot\|_{H_{\cT}}=\|\cdot\|_{H_{\cT},\kappa}$
(see \eqref{eq:norm-H-T-kappa}) is a norm on the discrete gradients.
By the discrete Poincaré inequality, \eqref{eq:general-conv-order}
also implies an convergence estimate for the discrete solutions itself.
The theorem can be understood as a requirement on the regularity of
$u$, resp. the right hand side of \eqref{eq:abstract linear}.

The combination of the proofs of Theorems 27 and 33 of \cite{di2018third}
shows that the variational consistency error for
\[
a\left(u,v\right)=\int_{\Omega}\nabla u\cdot\kappa\nabla v\,,\qquad a_{\cT}\left(u,v\right)=\sum_{i\sim j}\frac{m_{ij}}{h_{ij}}\left(u_{j}-u_{i}\right)\kappa_{ij}\left(v_{j}-v_{i}\right)
\]
becomes
\begin{equation}
\mathfrak{E}_{\cT}\left(u;v\right)=\sum_{i\sim j}\left(v_{j}-v_{i}\right)\left(\int_{\sigma_{ij}}\kappa\nabla u\cdot\bnu_{ij}-\frac{m_{ij}}{h_{ij}}\kappa_{ij}\left(\left(\cR_{\cT}u\right)_{j}-\left(\cR_{\cT}u\right)_{i}\right)\right)\,.\label{eq:var-cons-err-FV}
\end{equation}
Introducing on $L^{2}\left(\cP\right)$ the $H_{\cT}$-norm given
by
\begin{equation}
\norm u_{H_{\cT,\kappa}}:=\sum_{i\sim j}\frac{m_{ij}}{h_{ij}}\kappa_{ij}\left(u_{j}-u_{i}\right)^{2},\label{eq:norm-H-T-kappa}
\end{equation}
we find
\begin{equation}
\norm{\mathfrak{E}_{\cT}\left(u;\cdot\right)}_{H_{\cT,\kappa}^{\ast}}\leq\sum_{i\sim j}\frac{h_{ij}}{m_{ij}}\kappa_{ij}^{-1}\left(\int_{\sigma_{ij}}\kappa\nabla u\cdot\bnu_{ij}-\frac{m_{ij}}{h_{ij}}\kappa_{ij}\left(\left(\cR_{\cT}u\right)_{j}-\left(\cR_{\cT}u\right)_{i}\right)\right)^{2}\,.\label{eq:dual-norm-var-cons-err}
\end{equation}
In view of the Poincaré inequality in Theorem \ref{thm:discrete-PI}
the norm $\norm{\cdot}_{L^{2}\left(\cP\right)}$ is bounded by $\norm{\cdot}_{H_{\cT,\kappa}}$
in case $\kappa$ is uniformly bounded away from $0$. The right hand
side of equation \eqref{eq:dual-norm-var-cons-err} gives rise to
the definition of a ``dual'' $H_{\cT}$-norm which we denote
\begin{equation}
\norm u_{H_{\cT,\kappa}^{-}}:=\sum_{i\sim j}\frac{m_{ij}}{h_{ij}}\kappa_{ij}\left(u_{j}-u_{i}\right)^{2}.\label{eq:norm-H-T-kappa-}
\end{equation}

With regard to \eqref{eq:general-conv-order} and Lemma \ref{lem:Gal-Herb-Vign},
the above considerations motivate the following definition. 
\begin{defn}[$\varphi$-consistency]
\label{def:consistent} Let $\cT_{h}=\left(\mathcal{V}_{h},\cE_{h},\cP_{h}\right)$
be a family of meshes with $\diam\cT_{h}\to0$ as $h\to0$. We say
that $\cT_{h}$ is $\varphi$-consistent (satisfies $\varphi$-consistency)
on the subspace $B\subset H_{0}^{1}\left(\Omega\right)$ if for every
$u\in B$ there exists $C\geq0$ such that for every $h>0$
\[
\sum_{\sigma_{ij}\in\cE_{h}}\frac{h_{ij}}{m_{ij}}\kappa_{ij}^{-1}\left|\int_{\sigma_{ij}}\kappa\nabla u\cdot\bnu_{ij}-\kappa_{ij}\frac{m_{ij}}{h_{ij}}\left(\left(\cR_{\cT_{h}}u\right)_{j}-\left(\cR_{\cT_{h}}u\right)_{i}\right)\right|^{2}\leq C\varphi\left(h\right)^{2}\,.
\]
\end{defn}

Hence, we immediately obtain the following.
\begin{prop}
Under the assumptions of Lemma \ref{lem:Gal-Herb-Vign} and assuming
$h_{ij}\leq Ch$ for some constant $C>0$ the mesh is $\varphi$-consistent
with $\varphi(h)=h$. We say that the mesh is $h$-consistent.
\end{prop}

\subsection{\textit{\emph{\label{subsec:Consistency-on-cubic}Consistency on
cubic meshes}}}

For $d\leq3$, we consider a polygonal domain $\Omega\subset\R^{d}$
with a cubic mesh where $\Omega_{i}=x_{i}+[-h/2,h/2]^{d}$, $x_{i}\in h\mathbb{Z}\subset\Omega$,

Then we want to estimate the terms
\begin{align*}
m_{ij}\left|J_{ij}^{S}\hat{U}-\overline{J}_{ij}^{\star}U\right| & =S_{ij}\left|m_{ij}\kappa_{ij}\frac{\hat{U_{j}}-\hat{U}_{i}}{h}-\int_{\sigma_{ij}}\kappa\nabla U\cdot\nu_{ij}\right|\\
 & \leq S_{ij}\left|m_{ij}\kappa_{ij}\frac{\hat{U_{j}}-\hat{U}_{i}}{h}-\kappa_{ij}\int_{\sigma_{ij}}\nabla U\cdot\nu_{ij}\right|+S_{ij}\left|\int_{\sigma_{ij}}\left(\kappa_{ij}-\kappa\right)\nabla U\cdot\nu_{ij}\right|.
\end{align*}

In fact the following calculations are quite standard and, therefore,
we shorten our considerations.

Now, we want to estimate $\left|m_{ij}\frac{\hat{U_{j}}-\hat{U}_{i}}{h}-\int_{\sigma_{ij}}\nabla U\cdot\nu_{ij}\right|.$
We have $\hat{U}_{j}=U(x)+\nabla U\cdot(x_{j}-x)+O(h^{2})$ and $\hat{U}_{i}=U(x)+\nabla U\cdot(x_{i}-x)+O(h^{2})$.
Moreover, we can write $x_{i}-x=-\frac{h}{2}\nu_{ij}+\tilde{x}$ where
$\tilde{x}\perp\nu_{ij}$ and $x_{j}-x=\frac{h}{2}\nu_{ij}+\tilde{x}$
(the normal $\nu_{ij}$ points outside or inside of $\Omega_{i}$).
Hence, we conclude
\begin{align*}
\hat{U}_{j} & =U(x)+\nabla U\cdot(\frac{h}{2}\nu_{ij}+\tilde{x})+O(h^{2})\\
\hat{U}_{i} & =U(x)+\nabla U\cdot(-\frac{h}{2}\nu_{ij}+\tilde{x})+O(h^{2}).
\end{align*}
Subtracting both equations, we end up with $\frac{\hat{U}_{j}-\hat{U}_{i}}{h}=\nabla U\cdot\nu_{ij}+O(h^{2})$,
and hence,
\[
\left|m_{ij}\frac{\hat{U_{j}}-\hat{U}_{i}}{h}-\int_{\sigma_{ij}}\nabla U\cdot\nu_{ij}\right|\leq m_{ij}O(h^{2}).
\]

\begin{thm}[Consistency on cubic meshes]
\label{thm:ConsistencyCubicGrid} Let $\Omega\subset\R^{d}$ with
$d\leq3$ be a polygonal domain with a cubic mesh where $\Omega_{i}=x_{i}+[-h/2,h/2]^{d}$,
$x_{i}\in h\mathbb{Z}\subset\Omega$ and let $\kappa\equiv1$. Then
\[
\norm{\fE_{\cT}\left(u;\cdot\right)}_{H_{\cT}^{\ast}}\leq Ch^{2}.
\]
\end{thm}

We will consider a general $\kappa$ in Section \ref{subsec:Qualitative-comparison-on-CM}
below.

\section{Derivation of the methods and formal comparison}

In this section, we recall the original derivation of the Scharfetter--Gummel
scheme and then show that both the SG and the SQRA scheme are members
of a huge family of discretization schemes. Finally, we provide a
physically motivated derivation of the SQRA scheme.

\subsection{\label{subsec:Motivation-of-SG}Motivation of the Scharfetter--Gummel
scheme}

\paragraph*{One dimensional case}

The Scharfetter--Gummel scheme for the discrete flux on the interval
$[0,h]$ is derived under the assumption of constant flux $J$, force
$q=-\mathrm{d}V/\mathrm{d}x$ and diffusion coefficient $\kappa$
on $[0,h]$. We consider the two-point boundary value problem
\begin{equation}
J =-\kappa\left(\frac{\mathrm{d}u}{\mathrm{d}x}-qu\right)\quad\text{on }[0,h],
\qquad u\left(0\right) =u_{0}, 
\qquad u\left(h\right) =u_{h},\label{eq:Ex-Weight-Flux-scheme}
\end{equation}

where $q:[0,h]\rightarrow\R$ describes the constant force inducing
the drift current (i.e., the potential $V$ is a linear function on
$[0,h]$) and $\kappa>0$ is a constant, positive diffusion coefficient.
The problem has an elementary solution of the form
\[
u\left(x\right)=\frac{J}{\kappa q}+\left(u_{0}-\frac{J}{\kappa q}\right)\e^{qx}.
\]
Using the second boundary value $u(h)=u_{h}$, we get an explicit
form for the flux
\begin{equation}
J=\frac{\kappa}{h}\left(u_{0}B\left(-qh\right)-u_{h}B\left(qh\right)\right),\label{eq:J-for-SG}
\end{equation}
where $B(r)=r/\left(\e^{r}-1\right)$ is the Bernoulli function. Finally,
using $q=-\left(V_{h}-V_{0}\right)/h$, we can equally write 
\begin{equation}
J=-\frac{\kappa}{h}\frac{V_{h}-V_{0}}{\e^{V_{h}}-\e^{V_{0}}}\left(\frac{u_{h}}{\pi_{h}}-\frac{u_{0}}{\pi_{0}}\right).\label{eq:discr-flux}
\end{equation}

\paragraph*{Higher dimensional case}

In higher dimensions, the flux between two neighboring cells $j\sim i$
is discretized along the same lines as in the one dimensional case
(i.e., assumption of constant force, flux and diffusion constant along
each edge of the mesh). We project the flux $\mathbf{J}$ on the edge
$\mathbf{h}_{ij}:=x_{j}-x_{i}$
\[
\mathbf{h}_{ij}\cdot\mathbf{J}=-\kappa_{ij}\left(\mathbf{h}_{ij}\cdot\nabla u+u\mathbf{h}_{ij}\cdot\nabla V\right),
\]
where the assumption of a linear affine potential (inducing the constant
force $\mathbf{q}=-\nabla V$) implies that $\mathbf{h}_{ij}\cdot\nabla V=V_{i}-V_{j}$.
Moreover, we write $u\left(x\right)=u\left(x\left(s\right)\right)$
with $x\left(s\right)=sx_{i}+\left(1-s\right)x_{j}$,
where $0\leq s\leq1$ parametrizes the position on the edge. Then,
with $\mathrm{d}u/\mathrm{d}s=\mathbf{h}_{ij}\cdot\nabla u$ and $\mathbf{h}_{ij}\cdot\mathbf{J}=h_{ij}J_{ij}$,
we arrive at the two-point boundary problem
\[
h_{ij}J_{ij}=-\kappa_{ij}\left(\frac{\mathrm{d}u}{\mathrm{d}s}+u\left(V_{i}-V_{j}\right)\right)\quad\text{on }s\in\left[0,1\right],\qquad u\left(0\right)=u_{j},\qquad u\left(1\right)=u_{i},
\]
which is equivalent to the one dimensional problem.~\eqref{eq:Ex-Weight-Flux-scheme}.
The solution reads
\[
J_{ij}=\frac{\kappa_{ij}}{h_{ij}}\left(u_{j}B\left(V_{i}-V_{j}\right)-u_{i}B\left(-\left(V_{i}-V_{j}\right)\right)\right),
\]
which can also be written as
\[
J_{ij}=-\frac{\kappa_{ij}}{h_{ij}}\frac{V_{i}-V_{j}}{e^{V_{i}}-e^{V_{j}}}\left(\frac{u_{i}}{\pi_{i}}-\frac{u_{j}}{\pi_{j}}\right).
\]

\begin{rem}
In case of a sufficiently fine discretization that accurately takes
into account the structure of $V$, we can expect that $\left|V_{j}-V_{i}\right|\ll1$
is small such that $\frac{V_{i}-V_{j}}{\e^{V_{i}}-\e^{V_{j}}}\approx\sqrt{\pi_{i}\pi_{j}}+O\left(\pi_{i}-\pi_{j}\right)^{2}$.
We then infer 
\[
J_{ij}=-\frac{\kappa_{ij}}{h_{ij}}\sqrt{\pi_{i}\pi_{j}}\left(\frac{u_{j}}{\pi_{j}}-\frac{u_{i}}{\pi_{i}}\right)\,,
\]
which is the flux discretization according to the SQRA scheme. This
becomes more clear in the following sections.
\end{rem}

\begin{rem}[Motivation of discretized diffusion coefficient $\kappa_{ij}$]

Considering inhomogeneous media, where the diffusion coefficient is
not necessarily constant, a suitable discretization for the diffusion
$\kappa_{ij}$ is needed. Let us neglect for a moment the drift $\mathbf{q}u$
and assume that we have $\kappa_{i}$ on $\Omega_{i}$ around $x_{i}$
and $\kappa_{j}$ on $\Omega_{j}$ around $x_{j}$. We compute the
flux $J_{ij}$ from $x_{i}$ to $x_{j}$. Let $x_{0}$ be the intersection
of $\mathbf{h}_{ij}$ and $\sigma_{ij}$, and moreover, $d_{i}=|x_{0}-x_{i}|$
and $d_{j}=|x_{j}-x_{0}|$. The density at $x_{0}$ is denoted by
$u_{0}$. The flux $J_{i}$ from $x_{i}$ to $x_{0}$ is then given
by $J_{i}=-\kappa_{i}\frac{u_{0}-u_{i}}{d_{i}}$ and the flux $J_{j}$
from $x_{0}$ to $x_{j}$ is then given by $J_{j}=-\kappa_{j}\frac{u_{j}-u_{0}}{d_{j}}$.
The flux from $x_{i}$ to $x_{j}$ is then given by $J_{ij}=-\kappa_{ij}\frac{u_{j}-u_{i}}{d_{i}+d_{j}}$.
Hence, we have

\[
J_{ij}=-\frac{\kappa_{ij}}{d_{i}+d_{j}}\left(u_{j}-u_{0}+u_{0}-u_{i}\right)=\frac{\kappa_{ij}}{d_{i}+d_{j}}\left(\frac{d_{i}}{\kappa_{i}}J_{i}+\frac{d_{j}}{\kappa_{j}}J_{j}\right).
\]

Kirchhoff's law says $J_{i}=J_{j}=J_{ij}$, which implies that $1=\frac{\kappa_{ij}}{d_{i}+d_{j}}\left(\frac{d_{i}}{\kappa_{i}}+\frac{d_{j}}{\kappa_{j}}\right)$
and hence the weighted harmonic mean
\[
\kappa_{ij}=\frac{(d_{i}+d_{j})}{\frac{d_{i}}{\kappa_{i}}+\frac{d_{j}}{\kappa_{j}}}=\frac{1}{\frac{1}{\kappa_{i}}\frac{d_{i}}{d_{i}+d_{j}}+\frac{1}{\kappa_{j}}\frac{d_{j}}{d_{i}+d_{j}}}.
\]

Note that $1/\kappa$ is the mobility and hence we conclude $\frac{1}{\kappa_{ij}}=\frac{1}{\kappa_{i}}\frac{d_{i}}{d_{i}+d_{j}}+\frac{1}{\kappa_{j}}\frac{d_{j}}{d_{i}+d_{j}}$,
i.e. the arithmetic mean of the mobilities $1/\kappa_{i}$ and $1/\kappa_{j}$.

Interestingly, the harmonic mean is yet another special case of Stolarsky
means (see below) for $\alpha=-2$ and $\beta=-1$. Thus classical
FV discretizations of classical elliptic problems based on discretizations
of $-\Delta$ are another particular case of our general study.
\end{rem}

\subsection{\label{subsec:A-family-of-schemes}A family of discretization schemes}

Repeating the above calculations from a different point of view reveals
some additional structure of the Scharfetter--Gummel scheme and puts
it into a broader context.

Taking into account the special structure of the Fokker--Planck equation
in \eqref{eq:Ex-Weight-Flux-scheme}, we solve
\begin{align*}
\frac{1}{\kappa}J & =-\left(u'\left(x\right)+u\left(x\right)V'\left(x\right)\right),\qquad u(0)=u_{0},\ u(h)=u_{h},
\end{align*}
for a general potential $V:[0,h]\rightarrow\R$ not necessarily assumed
to be affine. The general solution reads
\[
u(x)=-\left(\frac{1}{\kappa}J\int_{0}^{x}\e^{V}+u_{0}\e^{V_{0}}\right)\e^{-V(x)}.
\]
The flux can be computed explicitly from the assumption $J=\mathrm{const}.$
and setting $x=h$ in the above formula. This yields
\[
J=-\kappa\frac{u_{h}\e^{V_{h}}-u_{0}\e^{V_{0}}}{\int_{0}^{h}\e^{V}}=-\kappa\frac{1}{h}\left(\frac{1}{h}\int_{0}^{h}\pi^{-1}\right)^{-1}\left(\frac{u_{h}}{\pi_{h}}-\frac{u_{0}}{\pi_{0}}\right)=-\kappa\pi_{\mathrm{mean}}\frac{1}{h}\left(\frac{u_{h}}{\pi_{h}}-\frac{u_{0}}{\pi_{0}}\right)
\]
for the averaged $\pi_{\mathrm{mean}}=\left(\frac{1}{h}\int_{0}^{h}\pi^{-1}\right)^{-1}$,
which clearly determines the constant flux along the edge. In particular,
assuming that $V$ is affine, i.e. $V(x)=\frac{V_{h}-V_{0}}{x_{h}-x_{0}}\left(x-x_{0}\right)+V_{0}$,
one easily checks that $\pi_{\mathrm{mean}}=\left(V_{h}-V_{0}\right)/\left(\e^{V_{h}}-\e^{V_{0}}\right)$,
which is the mean corresponding to the Scharfetter--Gummel discretization.
However, a potential can also be approximated not by piecewise affine
interpolation but in other ways, resulting in different means $\pi_{\mathrm{mean}}$.
We provide an example of such an approximation for the SQRA in the
Appendix~\ref{subsec:Approximation-of-potential}.

We aim to express $\pi_{\mathrm{mean}}$ by means of the values $\pi_{0}$
and $\pi_{h}$ at the boundaries. The choice of this average is non-trivial
and determines the quality of the discretization
scheme, as we will see below. In the present work, we focus on the
(weighted) Stolarsky mean, although there are also other means like
general $f$-means ($M_{f}(x,y)=f\left(\frac{f^{-1}(x)+f^{-1}(y)}{2}\right)$
for a strictly increasing function $f$). The Stolarsky mean has the
advantage that it is a closed formula for a broad family of popular
means and that its derivatives can be computed explicitly.

The weighted Stolarsky mean $S_{\alpha,\beta}$ \cite{Stolarsky1975mean}
is given as
\begin{equation}
S_{\alpha,\beta}(x,y)=\left(\frac{\beta(x^{\alpha}-y^{\alpha})}{\alpha(x^{\beta}-y^{\beta})}\right)^{\frac{1}{\alpha-\beta}}\,,\label{eq:Stolarsky-means}
\end{equation}
whenever these expressions are well defined and continuously extended
otherwise. We note the symmetry properties $S_{\alpha,\beta}\left(x,y\right)=S_{\alpha,\beta}\left(y,x\right)=S_{\beta,\alpha}\left(x,y\right)$.
Interesting special limit cases are $S_{0,1}(x,y)=\frac{x-y}{\log\left(x/y\right)}=\Lambda(x,y)$
(logarithmic mean), $S_{-1,1}(x,y)=\sqrt{xy}$ (geometric mean) and
$S_{0,-1}(x,y)=\frac{xy}{\Lambda(x,y)}$ (Scharfetter--Gummel mean).
A list of further Stolarsky means is given in Table~\ref{tab:Choices-for-S}.

An explicit calculation shows that $\partial_{x}^{2}S_{0,-1}\left(x,x\right)=-\left(3x\right)^{-1}$
and $\partial_{x}^{2}S_{-1,1}\left(x,x\right)=-\left(4x\right)^{-1}$.
For the general Stolarsky mean $S_{\alpha,\beta}$ one obtains (see
Appendix \ref{sec:Verification-of-eq:Stolarsky-second-deriv})
\begin{equation}
\partial_{x}^{2}S_{\alpha,\beta}\left(x,x\right)=\frac{1}{12x}\left(\alpha+\beta-3\right)\,,\label{eq:Stolarsky-second-deriv}
\end{equation}
particularly reproducing the above findings for $\partial_{x}^{2}S_{0,-1}$
and $\partial_{x}^{2}S_{-1,1}$. With respect to \eqref{eq:bernoulli-B}--\eqref{eq:SQRA-B},
we observe that 
\begin{align}
B_{1}\left(V_{i}-V_{j}\right) & =\frac{V_{i}-V_{j}}{\e^{V_{i}-V_{j}}-1}=S_{0,-1}\left(\pi_{i},\pi_{j}\right)\pi_{j}^{-1}\,,\label{eq:B-vs-S-SG}\\
B_{2}\left(V_{i}-V_{j}\right) & =\e^{-\frac{1}{2}\left(V_{i}-V_{j}\right)}=S_{-1,1}\left(\pi_{i},\pi_{j}\right)\pi_{j}^{-1}\,,\label{eq:B-vs-S-SQRA}
\end{align}
and \eqref{eq:FP-discr-Bernoulli} can be brought into the form \eqref{eq:Fokker-Planck-SQRA-residual},
which we equally write as 
\begin{equation}
-\sum_{j:\,j\sim i}\frac{m_{ij}}{h_{ij}}\kappa_{ij}S_{\ast}\left(\pi_{i},\pi_{j}\right)\left(\frac{u_{j}^{\cT}}{\pi_{j}}-\frac{u_{i}^{\cT}}{\pi_{i}}\right)=f_{i}\,,\label{eq:Stolarsky flow}
\end{equation}
where $S_{\ast}$ equals either $S_{0,-1}$ or $S_{-1,1}$. For general
means $S_{\alpha,\beta}\left(x,y\right)$, we have the relation
$S_{\alpha,\beta}\left(x,y\right)=x S_{\alpha,\beta}\left(1,y/x\right)$,
such that the weight function for arbitrary parameters $\alpha$ and
$\beta$ reads
\[
B_{\alpha,\beta}\left(x\right)=S_{\alpha,\beta}\left(1,\e^{-x}\right).
\]
In particular, it holds for any $\alpha$ and $\beta$
\[
B_{\alpha,\beta}\left(-x\right)=\e^{x}B_{\alpha,\beta}\left(x\right),
\]
which guarantees the consistency of the scheme with the thermodynamic
equilibrium.

Interestingly, the derivation of the SQRA in Section 2.2 of \cite{LieFackWeb2013}
relies on the assumption that the flux through a FV-interface has
to be proportional to $\left(u_{j}^{\cT}/\pi_{j}-u_{i}^{\cT}/\pi_{i}\right)$
with the proportionality factor given by a suitable mean of $\pi_{i}$
and $\pi_{j}$. The choice of $S_{-1,1}$ in \cite{LieFackWeb2013}
seems arbitrary, yet it yields very good results \cite{weber2017fuzzy,fackeldey2019metastable,DonHeiWebKel2017}.

\begin{table}[h]
%
\renewcommand{\arraystretch}{1.7}
\begin{tabular}{>{\centering}m{5.19cm} >{\centering}m{4.2cm} >{\centering}m{1.5cm} >{\centering}m{1.5cm} >{\centering}m{1.5cm}}
\toprule
mean & weight $B(x)$ & $\alpha$ & $\beta$ & $\alpha+\beta$\tabularnewline
\hline 
max & $\begin{cases}
\e^{-x}, & x\leq0\\
1, & x>0
\end{cases}$ & $+\infty$ & $1$ & $+\infty$\tabularnewline
quadratic mean & $\sqrt{\frac{1}{2}\left(1+\e^{-2x}\right)}$ & $4$ & $2$ & $6$\tabularnewline
arithmetic mean & $\frac{1}{2}\left(1+\e^{-x}\right)$ & $2$ & $1$ & $3$\tabularnewline
logarithmic mean & $\frac{1}{x}\left(1-\e^{-x}\right)$ & $0$ & $1$ & $1$\tabularnewline
geometric mean (SQRA) & $\e^{-x/2}$ & $-1$ & $1$ & $0$\tabularnewline
Scharfetter--Gummel mean & $x/\left(\e^{x}-1\right)$ & $0$ & $-1$ & $-1$\tabularnewline
harmonic Mean & $2/\left(\e^{x}+1\right)$ & $-2$ & $-1$ & $-3$\tabularnewline
min & $\begin{cases}
\e^{x}, & x\leq0\\
1, & x>0
\end{cases}$ & $-\infty$ & $1$ & $-\infty$\tabularnewline
\bottomrule
\end{tabular}
\caption{\label{tab:Choices-for-S}Examples for popular mean values expressed
as (weighted) Stolarsky means $S_{\alpha,\beta}$ with corresponding 
weight functions in \eqref{eq:FP-discr-Bernoulli} that generalize the Bernoulli function. The geometric
mean corresponds to the SQRA, the $S_{0,-1}$-mean to the Scharfetter--Gummel
discretization.}
\label{tab: stolarsky means}
\end{table}

\subsection{\label{subsec:The-Wasserstein-gradient}The Wasserstein gradient
structure of the Fokker--Planck operator and the SQRA method}

The choice of $S_{\ast}$ turns out to be crucial for the convergence
properties. In this section, we look at physical structures which
are desirable to be preserved in the discretization procedure. Our
considerations are based on the variational structure of the Fokker--Planck
equation. Let us note at this point that a physically reasonable discretization
is not necessarily the best from the rate of convergence point of
view. Indeed, this last point will be underlined by numerical simulations
in Section~\ref{sec:Interpretation-and-simulation}. However, the
physical consideration is helpful to understand the family of Stolarsky
discretizations from a further, different point of view.

In \cite{jordan1998variational} it was proved that the Fokker--Planck
equation
\begin{equation}
\dot{u}=\nabla\cdot\left(\kappa\nabla u+\kappa u\nabla V\right)\label{eq:FP-Gradflow}
\end{equation}
has the gradient flow formulation $\dot{u}=\partial_{\xi}\Psi^{*}(u,-\rmD E(u))$
where
\begin{equation}
E(u)=\int_{\Omega} \left(u\log u+Vu-u+1\right)=\int_{\Omega} \left(u\log\left(\frac{u}{\pi}\right)-u+1\right)\,,\qquad\Psi^{\ast}(u,\xi)=\frac{1}{2}\int_{\Omega}\kappa u\left|\nabla\xi\right|^{2}\,,\label{eq:FP-Gradflow-E-Psi}
\end{equation}
and $\pi=\e^{-V}$ is the stationary solution of \eqref{eq:FP-Gradflow}.
Indeed, one easily checks that $\rmD E(u)=\log u+V=\log\left(\frac{u}{\pi}\right)$
and $\partial_{\xi}\Psi^{\ast}(u,\xi)=-\nabla\cdot(\kappa u\nabla\xi)$
such that it formally holds
\[
\partial_{\xi}\Psi^{\ast}(u,\xi)|_{\xi=-\mathrm{D}E(u)}=-\nabla\cdot(\kappa u\nabla\xi)|_{\xi=-\mathrm{D}E(u)}=\nabla\cdot\left(\kappa u\left(\frac{\nabla u}{u}+\nabla V\right)\right)=\nabla\cdot\left(\kappa\nabla u+\kappa u\nabla V\right)=\dot{u}.
\]
Given a particular partial differential equation, the gradient structure
might not be unique. For example, the simple parabolic equation $\partial_{t}u=\Delta u$
can be described by \eqref{eq:FP-Gradflow-E-Psi} with $V=0$. But
at the same time one might choose $E(u)=\int u^{2}$ with $\Psi^{\ast}\left(\xi\right)=\int\left|\nabla\xi\right|^{2}$,
which plays a role in phase field modeling (see \cite{HeidaMalekRajagopal2011b}
and references therein) or $E(u)=-\int\log u$ with $\Psi^{\ast}\left(\xi\right)=\int u^{2}\left|\nabla\xi\right|^{2}$.

In view of this observation, one might pose the question about ``natural''
gradient structures of the discretization schemes. This is reasonable
if one believes that discretization schemes should incorporate the
underlying physical principles. The energy functional is clearly prescribed
by \eqref{eq:FP-Gradflow-E-Psi} with the natural discrete equivalent
\begin{equation}
E_{\cT}(u)=\sum_{i}m_{i}\left(u_{i}\log\left(\frac{u_{i}}{\pi_{i}}\right)-u_{i}+1\right)\,.\label{eq:discr-gf-1}
\end{equation}
The discrete linear evolution equation can be expected to be linear.
Since we identified the continuous flux to be $\mathbf{J}=-\kappa\pi\nabla U$
with $U=u/\pi$, we expect the form 
\begin{equation}
\dot{u}_{i}m_{i}=\partial_{\xi}\Psi_{\cT}^{*}(u,-\rmD E_{\cT}(u))=\sum_{j:i\sim j}\frac{m_{ij}}{h_{ij}}\kappa_{i,j}\pi_{ij}\left(\frac{u_{j}}{\pi_{j}}-\frac{u_{i}}{\pi_{i}}\right)\label{eq:Lin-evol-eq-a}
\end{equation}
for some suitably averaged $\pi_{ij}$. Equation \eqref{eq:Lin-evol-eq-a}
can be understood as a time-reversible (or detailed balanced) Markov
process on the finite state space ${\mathcal P}$. Recently, various different
gradient structures have been suggested for \eqref{eq:Lin-evol-eq-a}:
\cite{Miel11GSRD,Maas11GFEF,ErbMaa12RCFM,CHLZ12FPEF,Miel13GCRE} for
a quadratic dissipation as a generalization of the Jordan--Kinderlehrer--Otto
approach; and \cite{MiPeRe14RGFL,MPPR17NETP}, where a dissipation
of cosh-type was appeared in the Large deviation rate functional for
a hydrodynamic limit of an interacting particle system. All of them
can be written in the abstract form
\begin{equation}
\Psi_{\cT}^{\ast}(u,\xi)=\frac{1}{2}\sum_{i}\frac{1}{m_{i}}\sum_{j:i\sim j}\frac{m_{ij}}{h_{ij}}S_{ij}a_{ij}(u,\pi)\mathsf{\psi^{*}}\left(\xi_{i}-\xi_{j}\right)\,,\label{eq:S_ij--general-Psi}
\end{equation}
where 
\begin{equation}
a_{ij}(u,\pi)=\left(\frac{u_{i}}{\pi_{i}}-\frac{u_{j}}{\pi_{j}}\right)\partial_{\xi}\mathsf{\psi^{*}}\left(\log\left(\frac{u_{i}}{\pi_{i}}\right)-\log\left(\frac{u_{j}}{\pi_{j}}\right)\right)^{-1}\,.\label{eq:S_ij--general-a}
\end{equation}

In fact, any positive and convex function $\psi^{*}$ defines a reasonable
dissipation functional $\Psi^{*}$ by \eqref{eq:S_ij--general-Psi}
and \eqref{eq:S_ij--general-a}. A special case is when choosing for
$\psi^{*}$ and exponentially fast growing function $\psi^{*}(r):=\mathsf{C^{*}}(r):=2\left(\cosh(r/2)-1\right)$.
Then $a_{ij}$ simplifies to 
\[
a_{ij}(u,\pi)=\sqrt{\frac{u_{i}u_{j}}{\pi_{i}\pi_{j}}},
\]
and hence, the square root appears. Choosing $S_{ij}=\sqrt{\pi_{i}\pi_{j}}$,
we end up with a dissipation functional of the form

\begin{equation}
\Psi_{\cT}^{\ast}(u,\xi)=\sum_{i}\sum_{j:i\sim j}m_{ij}h_{ij}\sqrt{u_{i}u_{j}}\,\frac{1}{h_{ij}^{2}}\mathsf{C^{*}}\left(\xi_{i}-\xi_{j}\right)\,.\label{eq:discr-gf-2}
\end{equation}

There are (at least) three good reasons why choosing this gradient
structure, i.e., modeling fluxes in exponential terms: a historical,
a mathematical and a physical:
\begin{enumerate}
\item Already in Marcelin's PhD thesis from 1915 (\cite{Marc15CECP}) exponential
reaction kinetics have been derived, which are still common in chemistry
literature.
\item Recently, convergence for families of gradient systems has been derived
based on the energy-dissipation principle (the so-called EDP-convergence
\cite{Miel16EGCG,LMPR17MOGG,DoFrMi18GSWE}). Vice versa, the above
cosh-gradient structure appears as an effective gradient structure
applying EDP-convergence to Wasserstein gradient flow problems \cite{LMPR17MOGG,FreLie19?EDTS}.
\item Recalling the gradient structure for the continuous Fokker--Planck
equation \eqref{eq:FP-Gradflow-E-Psi}, we observe that the dissipation
mechanism $\Psi^{*}$ is totally independent of the particular form
of the energy ${\mathcal E}$, which is determined by the potential $V$.
This is physically understandable, since a change of the energy resulting,
e.g., from external fields should not influence the dissipation structure.
The same holds for the discretized version \eqref{eq:discr-gf-2}.
In fact it was shown in \cite{MieSte19ECLRS}, that the only discrete
gradient structure, where the dissipation does not depend on $V$
resp. $\pi=\e^{-V}$, is the cosh-gradient structure with the SQRA
discretization $S_{ij}=S_{-1,1}(\pi_{i},\pi_{j})$. In particular,
this characterizes the SQRA. For convenience, we add a proof for that
to the Appendix \ref{subsec:Proof-of-Theorem-Stolarsky-grad}.
\end{enumerate}
We think that these properties distinguish the SQRA, although in the
following the convergence proofs do not really rely on the particular
discretization weight $S_{ij}$.
\begin{rem}[Convergence of energy and dissipation functional]
Let us finally make some comments on the convergence of $E_{\mathcal{T}}$
and $\Psi_{\mathcal{T}}^{*}$ given in \eqref{eq:discr-gf-1} and
\eqref{eq:discr-gf-2} to the continuous analogies $E$ and $\Psi^{*}$.
$\Gamma$-convergence can be shown if the fineness of $\cT$ tends
to $0$. For the energies it is clear, since $u\mapsto u\log\left(u/\pi\right)-u$
is convex. For the dissipation potentials $\Psi_{\cT}^{\ast}(u,\xi)$
we observe the following: For smooth functions $u$ and $\xi$, we
have $\frac{1}{h_{ij}^{2}}\mathsf{C^{*}}\left(\xi_{i}-\xi_{j}\right)\approx\frac{1}{2}\left(\frac{x_{i}-x_{j}}{\left|x_{i}-x_{j}\right|}\cdot\nabla\xi\right)^{2}+O(h_{ij}^{2})$
and $\sqrt{u_{i}u_{j}}\approx u\left(\frac{1}{2}(x_{i}+x_{j})\right)$.
The considerations from Section \ref{subsec:Fluxes-and--spaces} then
yield $\Psi_{\cT}^{\ast}(u,\xi)\approx\frac{1}{2}\int_{\bQ}u\left|\nabla\xi\right|^{2}$.

For quadratic dissipation, qualitative convergence results in 1-D
using the underlying gradient structure are obtained in \cite{DisLie15GSMC}
looking at energy-dissipation mechanism, and in \cite{GKMP19HODDOT}
proving convergence of the metric.
\end{rem}

\section{\label{sec:Comparison-of-discretization}Comparison of discretization
schemes}

We mutually compare any two discretization schemes of the form \eqref{eq:Fokker-Planck-SQRA-residual}
in case of Dirichlet boundary conditions. In this case, even though
the problem is only defined on $\tilde{\cP}$, we can simply sum over
all $\cP$ once we multiplied with a test function that assumes the
value $0$ at all $\cP\backslash\tilde{\cP}$.

Let us recall the formula \eqref{eq:fluxes} for the fluxes
\[
J_{ij}^{S}U=-\frac{\kappa_{ij}}{h_{ij}}S_{ij}(U_{j}-U_{i}).
\]
Moreover, let $u_{i}=U_{i}\pi_{i}$ and $\tilde{u_{i}}=\tilde{U}_{i}\pi_{i}$
be the solution of the discrete FPE \eqref{eq:Fokker-Planck-SQRA-residual}
for two different smooth mean coefficients $S_{ij}=S(\pi_{i},\pi_{j})$
and $\tilde{S}_{ij}=\tilde{S}(\pi_{i},\pi_{j})$ (e.g. once for Scharfetter--Gummel
and once for SQRA) such that 
\begin{align}
\sum_{k:k\sim i}m_{ik}h_{ik}J_{ik}^{S}U & =m_{i}\bar{f}_{i}\label{eq:comparison-1}\\
\sum_{k:k\sim i}m_{ik}h_{ik}J_{ik}^{\tilde{S}}\tilde{U} & =m_{i}\bar{f_{i}}.\label{eq:comparison-2}
\end{align}
In order to compare the solutions of \eqref{eq:comparison-1} and
\eqref{eq:comparison-2} we take the difference of these two equations
and multiply with $E_{i}=U_{i}-\tilde{U_{i}}$. We obtain
\begin{align*}
0 & =\sum_{i}\sum_{k:k\sim i}m_{ik}h_{ik}\left(J_{ik}^{S}U-J_{ik}^{\tilde{S}}\tilde{U}\right)E_{i}\\
 & =\sum_{i}\sum_{k:\,k\sim i}\frac{m_{ik}}{h_{ki}}\kappa_{ij}(S_{ik}(U_{i}-U_{k})-\tilde{S}_{ik}(\tilde{U_{i}}-\tilde{U_{k}}))E_{i}
\end{align*}

Introducing the notation $\alpha_{ik}=\kappa_{ik}\frac{m_{ik}}{h_{ik}}$
and using \eqref{eq:general-disc-partial-int} we get
\begin{align*}
0 & =\sum_{k\sim i}\alpha_{ik}\left(S_{ik}(U_{i}-U_{k})-S_{ik}(\tilde{U_{i}}-\tilde{U_{k}})+\left(S_{ik}-\tilde{S}_{ik}\right)(\tilde{U_{i}}-\tilde{U_{k}})\right)\left(E_{i}-E_{k}\right)\\
 & =\sum_{k\sim i}\alpha_{ik}\left(S_{ik}\left(E_{i}-E_{k}\right)+(S_{ik}-\tilde{S}_{ik})\left(\tilde{U_{i}}-\tilde{U_{k}}\right)\right)\left(E_{i}-E_{k}\right).
\end{align*}
Using the notation $\rmD_{ik}A=A_{k}-A_{i}$ for discrete gradients
\[
(\tilde{S}_{ik}-S_{ik})\left(\tilde{U_{i}}-\tilde{U_{k}}\right)\left(E_{i}-E_{k}\right)\leq\frac{1}{2}\left[S_{ik}\left(\rmD_{ik}E\right)^{2}+\frac{(S_{ik}-\tilde{S}_{ik})^{2}}{S_{ik}}\left(\rmD_{ik}\tilde{U}\right)^{2}\right]
\]
we get
\begin{align}
\frac{1}{2}\sum_{k\sim i}\alpha_{ik}S_{ik}\left(\rmD_{ik}E\right)^{2} & \leq\frac{1}{2}\sum_{k\sim i}\frac{(\tilde{S}_{ik}-S_{ik})^{2}}{S_{ik}\tilde{S}_{ik}}\alpha_{ik}\tilde{S}_{ik}\left(\rmD_{ik}\tilde{U}\right)^{2}\,.\label{eq:estimate-difference-1}
\end{align}
In the case of Stolarsky means the constants are more explicit. We
have the following expansion of $S_{ij}$: writing $\pi_{ij}=\frac{1}{2}\left(\pi_{i}+\pi_{j}\right)$,
$\pi_{+}=\pi_{-}=\frac{1}{2}\left(\pi_{i}-\pi_{j}\right)$ and $\pi_{i}=\pi_{0}+\pi_{+}$
and $\pi_{j}=\pi_{0}-\pi_{-}$
\begin{align}
S_{ij} & =S_{\alpha,\beta}\left(\pi_{ij},\pi_{ij}\right)+\frac{1}{2}\left(\pi_{+}-\pi_{-}\right)+\frac{1}{2}\partial_{x}^{2}S_{\alpha,\beta}\left(\pi_{ij},\pi_{ij}\right)\left(\pi_{+}+\pi_{-}\right)^{2}+O\left(\pi_{\pm}^{3}\right)\nonumber \\
 & =\pi_{ij}+\frac{\frac{1}{3}\left(\alpha+\beta\right)-1}{8\,\pi_{ij}}\left(\pi_{i}-\pi_{j}\right)^{2}+O\left(\pi_{i}-\pi_{j}\right)^{3}\,.\label{eq:Sij-approx}
\end{align}

In case $\left(\alpha+\beta\right)=\left(\tilde{\alpha}+\tilde{\beta}\right)$,
we obtain $S_{ij}-\tilde{S}_{ij}=O\left(\pi_{i}-\pi_{j}\right)^{3}$
and hence this yields the following first comparison result:
\begin{thm}
Let $\cT$ be a mehs with right hand side $f\in L^{2}(\mathcal{P})$
and let $u$ and $\tilde{u}$ be a two solution of the discrete FPE
for different Stolarsky mean coefficients $S_{ij}=S_{\alpha,\beta}\left(\pi_{i},\pi_{j}\right)$
and $\tilde{S}_{ij}=S_{\tilde{\alpha},\tilde{\beta}}\left(\pi_{i},\pi_{j}\right)$
respectively. Then 
\begin{multline*}
\frac{1}{2}\sum_{k\sim i}\kappa_{ik}\frac{m_{ik}}{h_{ik}}S_{ik}\left(\rmD_{ik}E\right)^{2}\\
\leq\frac{1}{2}\sum_{k\sim i}\left(\frac{\left(\left(\alpha+\beta\right)-\left(\tilde{\alpha}+\tilde{\beta}\right)\right)^{2}}{24^{2}\,\pi_{ij}^{2}\tilde{S}_{ik}S_{ik}}\left(\pi_{i}-\pi_{j}\right)^{4}+O\left(\pi_{i}-\pi_{j}\right)^{5}\right)\kappa_{ik}\frac{m_{ik}}{h_{ik}}\left(\rmD_{ik}\tilde{U}\right)^{2}
\end{multline*}
In case $\left(\alpha+\beta\right)=\left(\tilde{\alpha}+\tilde{\beta}\right)$
we furthermore find 
\[
\frac{1}{2}\sum_{k\sim i}\kappa_{ik}\frac{m_{ik}}{h_{ik}}S_{ik}\left(\rmD_{ik}E\right)^{2}\leq\frac{1}{2}\sum_{k\sim i}O\left(\pi_{i}-\pi_{j}\right)^{6}\kappa_{ik}\frac{m_{ik}}{h_{ik}}\left(\rmD_{ik}\tilde{U}\right)^{2}\,.
\]
\end{thm}

We aim to refine the above result to an order of convergence result
for $J^{S}U-J^{\tilde{S}}\tilde{U}$.. We introduce the auxiliary
smooth mean $\hat{S}_{ik}=\hat{S}(\pi_{i},\pi_{k})$ and find
\begin{align*}
\hat{S}_{ik} & \left(E_{i}-E_{k}\right)=\hat{S}_{ik}\left(U_{i}-\tilde{U}_{i}-\left(U_{k}-\tilde{U}_{k}\right)\right)\\
 & =S_{ik}(U_{i}-U_{k})-S_{ik}(U_{i}-U_{k})+\tilde{S}_{ik}(\tilde{U}_{i}-\tilde{U}_{k})-\tilde{S}_{ik}(\tilde{U}_{i}-\tilde{U}_{k})+\hat{S}_{ik}\left(U_{i}-\tilde{U}_{i}-\left(U_{k}-\tilde{U}_{k}\right)\right)\\
 & =m_{ik}\alpha_{ik}^{-1}\left(J_{ik}^{S}U-J_{ik}^{\tilde{S}}\tilde{U}\right)+\left(\hat{S}_{ik}-S_{ik}\right)\left(U_{i}-U_{k}\right)-\left(\hat{S}_{ik}-\tilde{S}_{ik}\right)\left(\tilde{U}_{i}-\tilde{U}_{k}\right)\,.
\end{align*}

Hence, we have
\begin{align*}
 & \sum_{k\sim i}\alpha_{ik}\left(S_{ik}(U_{i}-U_{k})-\tilde{S}_{ik}\left(\tilde{U_{i}}-\tilde{U_{k}}\right)\right)\left(E_{i}-E_{k}\right)\\
 & \qquad=\sum_{k\sim i}\frac{h_{ik}}{\kappa_{ik}}m_{ik}\frac{1}{\hat{S}_{ik}}\left(J_{ik}^{S}U-J_{ik}^{\tilde{S}}\tilde{U}\right)^{2}\\
 & \text{\qquad\quad+\ensuremath{\sum_{k\sim i}m_{ik}\frac{1}{\hat{S}_{ik}}\left(J_{ik}^{S}U-J_{ik}^{\tilde{S}}\tilde{U}\right)\left[\left(\hat{S}_{ik}-S_{ik}\right)\left(U_{i}-U_{k}\right)+\left(\hat{S}_{ik}-\tilde{S}_{ik}\right)\left(\tilde{U}_{i}-\tilde{U}_{k}\right)\right]\,},}
\end{align*}
and using Cauchy-Schwartz inequality, we get 
\begin{align*}
 & \sum_{k\sim i}\alpha_{ik}\left(S_{ik}(U_{i}-U_{k})-\tilde{S}_{ik}\left(\tilde{U_{i}}-\tilde{U_{k}}\right)\right)\left(E_{i}-E_{k}\right)\leq-\frac{1}{2}\sum_{k\sim i}\frac{h_{ik}m_{ik}}{\kappa_{ik}}\frac{1}{\hat{S}_{ik}}\left(J_{ik}^{S}U-J_{ik}^{\tilde{S}}\tilde{U}\right)^{2}\\
 & \qquad\qquad+\sum_{k\sim i}\frac{m_{ik}\kappa_{ik}}{h_{ik}\hat{S}_{ik}}\left(\left(\hat{S}_{ik}-S_{ik}\right)^{2}\left(U_{i}-U_{k}\right)^{2}+\left(\hat{S}_{ik}-\tilde{S}_{ik}\right)^{2}\left(\tilde{U}_{i}-\tilde{U}_{k}\right)^{2}\right)\,.
\end{align*}
Altogether we obtain 
\begin{align*}
\frac{1}{2}\sum_{k\sim i}\frac{h_{ik}m_{ik}}{\kappa_{ik}}\frac{1}{\hat{S}_{ik}}\left(J_{ik}^{S}U-J_{ik}^{\tilde{S}}\tilde{U}\right)^{2} & \leq\sum_{k\sim i}\frac{m_{ik}h_{ik}}{\kappa_{ik}\hat{S}_{ik}S_{ik}^{2}}\left(\hat{S}_{ik}-S_{ik}\right)^{2}\left(\frac{\kappa_{ik}}{h_{ik}}S_{ik}\left(U_{i}-U_{k}\right)\right)^{2}\\
 & +\sum_{k\sim i}\frac{m_{ik}h_{ik}}{\kappa_{ik}\hat{S}_{ik}\tilde{S}_{ik}^{2}}\left(\hat{S}_{ik}-\tilde{S}_{ik}\right)^{2}\left(\frac{\kappa_{ik}}{h_{ik}}\tilde{S}_{ik}\left(\tilde{U}_{i}-\tilde{U}_{k}\right)\right)^{2}\,.
\end{align*}

We make once more use of \eqref{eq:Sij-approx} writing $C_{\alpha,\beta}:=\frac{\alpha+\beta}{24}$
and exploiting $\pi_{i}=\pi_{ij}+\pi_{ij}\left(V_{i}-V_{ij}\right)+O\left(V_{i}-V_{ij}\right)^{2}$
with 
\begin{align*}
\pi_{i}-\pi_{j} & \approx\pi_{ij}\left(V_{i}-V_{j}\right)+O\left(V_{i}-V_{ij}\right)^{2}+O\left(V_{j}-V_{ij}\right)^{2}\\
S_{ij} & \approx\pi_{ij}+O\left(\pi_{i}-\pi_{j}\right)\,.
\end{align*}
Hence, we conclude the following result.
\begin{thm}
\label{thm:ComparisonDifferentDiscretizationSchemes}Let $\cT$ be
a mesh with right hand side $f\in L^{2}(\mathcal{P})$ and let $u$
and $\tilde{u}$ be two solutions of the discrete FPE for different
Stolarsky means $S$ and $\tilde{S}$. Moreover, let $\hat{S}$ be
any Stolarsky mean and assume that either $\alpha+\beta\neq\hat{\alpha}+\hat{\beta}$
or $\tilde{\alpha}+\tilde{\beta}\neq\hat{\alpha}+\hat{\beta}$. Then
the solutions $u$ and $\tilde{u}$ of the discretized FPE satisfy
the symmetrized error estimate up to higher order 
\begin{align*}
\frac{1}{2}\sum_{k\sim i}\frac{h_{ik}m_{ik}}{\kappa_{ik}}\frac{1}{\hat{S}_{ik}}\left(J_{ik}^{S}U-J_{ik}^{\tilde{S}}\tilde{U}\right)^{2} & \leq\sum_{k\sim i}\frac{m_{ik}h_{ik}}{\kappa_{ik}S_{ik}}\left(C_{\alpha,\beta}-C_{\hat{\alpha},\hat{\beta}}\right)\left(V_{i}-V_{j}\right)^{2}\left(J_{ik}^{S}U\right)^{2}\\
 & +\sum_{k\sim i}\frac{m_{ik}h_{ik}}{\kappa_{ik}\tilde{S}_{ik}}\left(C_{\tilde{\alpha},\tilde{\beta}}-C_{\hat{\alpha},\hat{\beta}}\right)\left(V_{i}-V_{j}\right)^{2}\left(J_{ik}^{\tilde{S}}\tilde{U}\right)^{2}\,.
\end{align*}
\end{thm}

More general, for any mean we have 
\begin{multline}
\frac{1}{2\kappa^{*}}\|J^{S}U-J^{\tilde{S}}\tilde{U}\|_{L_{\hat{S}}^{2}(\mathcal{E})}^{2}\\
\leq\frac{1}{\kappa_{*}}\left\{ \sup_{i,k}\frac{\left(\hat{S}_{ik}-S_{ik}\right)^{2}}{\hat{S}_{ik}S_{ik}}\left\Vert J^{S}U\right\Vert _{L_{S}^{2}(\mathcal{E})}^{2}+\sup_{i,k}\frac{\left(\hat{S}_{ik}-\tilde{S}_{ik}\right)^{2}}{\hat{S}_{ik}\tilde{S}_{ik}}\left\Vert J^{\tilde{S}}\tilde{U}\right\Vert _{L_{\tilde{S}}^{2}(\mathcal{E})}^{2}\right\} \,,\label{eq:ComparisonDifferentDiscretizationSchemes}
\end{multline}

and in particular for Stolarsky means with $\alpha+\beta=\tilde{\alpha}+\tilde{\beta}=\hat{\alpha}+\hat{\beta}$
we find the following result:%

\begin{cor}
\label{cor:ComparisonDifferenceForDifferentStolarskyMean}Let $\cT$
be a mesh with right hand side $f\in L^{2}(\mathcal{P})$ and let
$u$ and $\tilde{u}$ be two solutions of the discrete FPE for different
Stolarsky mean coefficients $S_{ij}=S_{\alpha,\beta}\left(\pi,\pi_{j}\right)$
and $\tilde{S}_{ij}=S_{\tilde{\alpha},\tilde{\beta}}\left(\pi,\pi_{j}\right)$
with $\alpha+\beta=\tilde{\alpha}+\tilde{\beta}=\hat{\alpha}+\hat{\beta}$.
Then estimate \eqref{eq:ComparisonDifferentDiscretizationSchemes}
holds. In particular, we find the refined estimate
\begin{align*}
 & \frac{1}{2\kappa^{*}}\|J^{S}U-J^{\tilde{S}}\tilde{U}\|_{L_{\hat{S}}^{2}(\mathcal{E})}^{2}\leq O\left(\pi_{i}-\pi_{j}\right)^{6}\left(\left\Vert J^{S}U\right\Vert _{L_{S}^{2}(\mathcal{E})}^{2}+\left\Vert J^{\tilde{S}}\tilde{U}\right\Vert _{L_{\tilde{S}}^{2}(\mathcal{E})}^{2}\right).
\end{align*}
\end{cor}

In particular, the last result shows that convergence rates are similar
up to order $3$ for different $\alpha,\beta$ which satisfy $\alpha+\beta=\mathrm{const}$.

\section{\label{sec:Convergence-of-the}Convergence of the discrete FPE}

In this section, we derive general estimates for the order of convergence
of the Stolarsky FV operators. Throughout this section, we assume
that the mesh satisfies the consistency property of Definition \ref{def:consistent}
with a suitable consistency function $\varphi:\R_{\geq0}\rightarrow\R_{\geq0}$
and discretization operator $\cR_{\cT}:H^{1}(\Omega)\supset B\rightarrow L^{2}(\cP)$.
The parameters $\pi_{i}$ are then given in terms of $\pi_{i}=\left(\cR_{\cT}\pi\right)_{i}$.

We derive consistency errors for $U$ in Section \eqref{subsec:Error-Analysis-in-U}
and consistency errors for $u$ in Section \eqref{subsec:Error-Analysis-in-u}.

\subsection{\label{subsec:Error-Analysis-in-U}Error Analysis in $U$}

In what follows, we assume that the discrete and the continuous solution
satisfy Dirichlet conditions. In view of the continuous and the discrete
FPE given in the form \eqref{eq:SQRA-conv-1} and \eqref{eq:SQRA-conv-2}
as well as formula \eqref{eq:var-cons-err-FV} we observe that the
natural variational consistency error for a given Stolarsky mean $S$
takes the form 
\[
\mathfrak{E}_{\cT,{\rm FPE}}\left(U;v\right)=\sum_{i\sim j}\left(v_{j}-v_{i}\right)\left(\int_{\sigma_{ij}}\kappa\pi\nabla U\cdot\bnu_{ij}-\kappa_{ij}S_{ij}\frac{m_{ij}}{h_{ij}}\left(\left(\cR_{\cT}U\right)_{j}-\left(\cR_{\cT}U\right)_{i}\right)\right)\,.
\]

We recall that an estimate for $\mathfrak{E}_{\cT,{\rm FPE}}\left(U;\cdot\right)$
implies an order of convergence estimate by \eqref{eq:general-conv-order}.
Our main result of this section provides a connection between $\mathfrak{E}_{\cT,{\rm FPE}}\left(U;\cdot\right)$
and the variational consistency $\mathfrak{E}_{\cT}\left(U;\cdot\right)$
(given by \eqref{eq:var-cons-err-FV}) of the second order equation
\[
-\nabla\cdot\left(\kappa\nabla U\right)=f
\]
with the discretization scheme 
\begin{align*}
\forall i:\qquad-\sum_{j:\,j\sim i}\kappa_{ij}\frac{m_{ij}}{h_{ij}}\left(U_{j}^{\cT}-U_{i}^{\cT}\right) & =f_{i}\,.
\end{align*}

\begin{prop}
\label{prop:conv-O-U-1}Let $\mathcal{T}=(\mathcal{V},\mathcal{E},\mathcal{P})$
be a mesh. The variational consistency error $\mathfrak{E}_{\cT,{\rm FPE}}\left(U;\cdot\right)$
can be estimated by
\begin{equation}
\norm{\mathfrak{E}_{\cT,{\rm FPE}}\left(U;\cdot\right)}_{H_{\cT,\kappa S}^{\ast}}^{2}\leq\norm{\pi}_{\infty}\norm{\mathfrak{E}_{\cT}\left(U;\cdot\right)}_{H_{\cT,\kappa}^{\ast}}^{2}+\sum_{i\sim j}\frac{h_{ij}}{m_{ij}}\kappa_{ij}^{-1}S_{ij}^{-1}\left(\int_{\sigma_{ij}}\left(\pi-S_{ij}\right)\kappa\nabla U\cdot\bnu_{ij}\right)^{2}\,.\label{eq:thm:conv-O-U-1}
\end{equation}
\end{prop}

\begin{proof}
For simplicity, we write $\hat{U}:=\cR_{\cT}U$. We observe that 
\begin{align*}
\mathfrak{E}_{\cT,{\rm FPE}}\left(U;v\right) & =-\sum_{i\sim j}\left(v_{j}-v_{i}\right)m_{ij}\left(\overline{J}_{ij}U-J_{ij}^{S}\hat{U}\right)\\
 & =-\sum_{i\sim j}\left(v_{j}-v_{i}\right)m_{ij}\left(\left(\overline{J}_{ij}U-\overline{J}_{ij}^{\star}U\right)+\left(\overline{J}_{ij}^{\star}U-J_{ij}^{S}\hat{U}\right)\right)\,,
\end{align*}
where 
\[
\overline{J}_{ij}^{\star}U:=-m_{ij}^{-1}\int_{\sigma_{ij}}\kappa S_{ij}\nabla U\cdot\bnu_{ij}\,.
\]
satisfies 
\begin{equation}
m_{ij}\left|\overline{J}_{ij}U-\overline{J}_{ij}^{\star}U\right|\leq\left|\int_{\sigma_{ij}}\left(\pi-S_{ij}\right)\kappa\nabla U\cdot\bnu_{ij}\right|\,.\label{eq:conv-rate-EFPE-1}
\end{equation}
Using the fact that 
\[
m_{ij}\left(\overline{J}_{ij}^{\star}U-J_{ij}^{S}\hat{U}\right)=-S_{ij}\left(\int_{\sigma_{ij}}\kappa\nabla U\cdot\bnu_{ij}-\kappa_{ij}\frac{m_{ij}}{h_{ij}}\left(\hat{U}_{j}-\hat{U}_{i}\right)\right)
\]
we obtain 
\begin{align}
\left|\sum_{i\sim j}\left(v_{j}-v_{i}\right)m_{ij}\left(\overline{J}_{ij}^{\star}U-J_{ij}^{S}\hat{U}\right)\right|\label{eq:conv-rate-EFPE-2}\\
\leq\norm v_{H_{\cT,\kappa S}}\left(\sup_{ij}S_{ij}\right)^{\frac{1}{2}} & \left(\sum_{i\sim j}\frac{h_{ij}}{m_{ij}}\kappa_{ij}^{-1}\left(\int_{\sigma_{ij}}\kappa\nabla U\cdot\bnu_{ij}-\kappa_{ij}\frac{m_{ij}}{h_{ij}}\left(\hat{U}_{j}-\hat{U}_{i}\right)\right)^{2}\right)^{\frac{1}{2}}\,.
\end{align}
From \eqref{eq:conv-rate-EFPE-1} we conclude 
\begin{equation}
\left|\sum_{i\sim j}\left(v_{j}-v_{i}\right)m_{ij}\left(\overline{J}_{ij}U-\overline{J}_{ij}^{\star}U\right)\right|\leq\norm v_{H_{\cT,\kappa S}}\left(\sum_{i\sim j}\frac{h_{ij}}{m_{ij}}\kappa_{ij}^{-1}S_{ij}^{-1}\left(\int_{\sigma_{ij}}\left(\pi-S_{ij}\right)\kappa\nabla U\cdot\bnu_{ij}\right)^{2}\right)^{\frac{1}{2}}\,.\label{eq:conv-rate-EFPE-3}
\end{equation}
Taking together \eqref{eq:conv-rate-EFPE-2}--\eqref{eq:conv-rate-EFPE-3}
we obtain \eqref{eq:thm:conv-O-U-1}.
\end{proof}
\begin{lem}
Assume there exists a constant $C>0$ such that for all cells $\Omega_{i},\Omega_{j}$
with $h_{i}=\diam\Omega_{i}$ it holds 
\begin{equation}
\norm f_{L^{2}\left(\sigma_{ij}\right)}^{2}\leq\frac{1}{h_{i}}C^{2}\norm f_{H^{1}\left(\Omega_{i}\right)}^{2}\,.\label{eq:embedd-bdrg-uniform-lemma}
\end{equation}
Then for $C^{2}$-smooth means $S$
\begin{equation}
\left|\int_{\sigma_{ij}}\left(\pi-S_{ij}\right)\kappa\nabla U\cdot\bnu_{ij}\right|\leq2C\left(m_{ij}h_{i}\right)^{\frac{1}{2}}\norm{\kappa\nabla U}_{H^{1}(\Omega_{i})}\,.\label{eq:EstimateMeanTerm}
\end{equation}
\end{lem}

\begin{rem}
Note that \eqref{eq:embedd-bdrg-uniform-lemma} can be easily verified
for cubes.
\end{rem}
\begin{proof}
Observe that

\begin{align}
\int_{\sigma_{ij}}\left|\pi-S_{ij}\right|\left|\kappa\nabla U\cdot\bnu_{ij}\right| & \leq\left(\int_{\sigma_{ij}}\left|\pi-S_{ij}\right|^{2}\right)^{\frac{1}{2}}\left(\int_{\sigma_{ij}}\left|\kappa\nabla U\cdot\bnu_{ij}\right|^{2}\right)^{\frac{1}{2}}\nonumber \\
 & \leq c\left(\int_{\sigma_{ij}}\left|\pi-S_{ij}\right|^{2}\right)^{\frac{1}{2}}\left(\frac{1}{h_{i}}\norm{\kappa\nabla U}_{H^{1}(\Omega_{i})}^{2}\right)^{\frac{1}{2}}\,.\label{eq:conv-rate-help-2}
\end{align}

It remains to study $\frac{1}{m_{ij}}\int_{\sigma_{ij}}\left|\pi-S_{ij}\right|^{2}$
in more detail. We have
\[
\pi-S_{ij}=\frac{1}{2}\left(\pi-\pi_{i}\right)+\frac{1}{2}\left(\pi-\pi_{j}\right)+\left(\frac{\pi_{i}+\pi_{j}}{2}-S_{ij}\right).
\]
The first term can be estimated by $|\pi-\pi_{i}|\leq h_{i}\cdot\nabla\pi+O(h_{i}^{2})$
and a similar estimate holds for the second term. The last term, assuming
that the mean is $C^{2}$-smooth, can be estimated by
\[
S(\pi_{i},\pi_{j})-S\left(\frac{\pi_{i}+\pi_{j}}{2},\frac{\pi_{i}+\pi_{j}}{2}\right)=\frac{1}{2}(\pi_{i}-\pi_{j})\nabla S\cdot(1,-1)^{T}+O(|\pi_{i}-\pi_{j}|).
\]
Using that $\pi_{i}-\pi_{j}=\nabla\pi\cdot h_{ij}+O(h_{ij})$ and
that $S\left(\frac{\pi_{i}+\pi_{j}}{2},\frac{\pi_{i}+\pi_{j}}{2}\right)=\frac{\pi_{i}+\pi_{j}}{2}$,
we obtain that $|\pi-S_{ij}|^{2}\leq O\left(h_{i}^{2}\right)$. In
total we obtain
\[
\int_{\sigma_{ij}}\left|\pi-S_{ij}\right|\left|\kappa\nabla U\cdot\bnu_{ij}\right|\leq2C\left(m_{ij}h_{i}^{2}\right)^{\frac{1}{2}}\left(\frac{1}{h_{i}}\norm{\kappa\nabla U}_{H^{1}(\Omega_{i})}^{2}\right)^{\frac{1}{2}}\,.
\]
\end{proof}
Using the above estimates, we can now show the main result of the
section.
\begin{thm}[Localized order of convergence]
\label{thm:ConsistencyError}Let the mesh $\cT$ be admissible in
sense of Definition \ref{def:admissible-grid} and consistent in sense
of Definition \ref{def:consistent}. Let $u\in C_{0}^{2}(\Omega)$
be the solution to \eqref{eq:Fokker-Planck-residual}. Let $f^{\cT}:=\cR_{\cT}^{\ast}f$
and let $u^{\cT}\in\cS^{\cT}$ be the solution to \eqref{eq:discrete-P-I}.
Moreover, let $\kappa\leq\kappa^{*}$, $b>0$ and $S\in C^{2}(\R_{\geq0}\times\R_{\geq0})$.
Then it holds it holds
\[
\|u^{{\mathcal T}}-{\mathcal R}_{{\mathcal T}}u\|_{H_{{\mathcal T}}}^2\leq C(\kappa_{*},\pi,d,\|U\|_{C^{2}})\times\left(\varphi(h)^{2}+h^{2}\right).
\]
\end{thm}

\begin{proof}
Inserting estimate \eqref{eq:EstimateMeanTerm} int to the estimate
of the variational consistency, we get
\begin{align*}
\norm{\mathfrak{E}_{\cT,{\rm FPE}}\left(U;\cdot\right)}_{H_{\cT,\kappa S}^{\ast}}^{2} & \leq\norm{\pi}_{\infty}\norm{\mathfrak{E}_{\cT}\left(U;\cdot\right)}_{H_{\cT,\kappa}^{\ast}}^{2}+C\sum_{i\sim j}h_{ij}\kappa_{ij}^{-1}S_{ij}^{-1}h_{i}\|\kappa\nabla U\|_{H^{1}(\Omega_{i})}^{2}\\
 & \leq\norm{\pi}_{\infty}\norm{\mathfrak{E}_{\cT}\left(U;\cdot\right)}_{H_{\cT,\kappa}^{\ast}}^{2}+C(\kappa_{*},\pi,d)\ h^{2}\sum_{i}\|\kappa\nabla U\|_{H^{1}(\Omega_{i})}^{2}.
\end{align*}
Using \eqref{eq:general-conv-order} we obtain an estimate for the
discretization error in the form
\begin{align*}
\|u^{{\mathcal T}}-{\mathcal R}_{{\mathcal T}}u\|_{H_{{\mathcal T}}}^{2} & \leq\norm{\pi}_{\infty}\norm{\mathfrak{E}_{\cT}\left(U;\cdot\right)}_{H_{\cT,\kappa}^{\ast}}^{2}+C(\kappa_{*},\pi,d,\|U\|_{C^{2}})\ \mathrm{Size}(\cT)^{2}.
\end{align*}
 Using the consistency assumption on the discretization of the pure
elliptic problem we obtain the desired estimate.
\end{proof}

\subsection{\label{subsec:Error-Analysis-in-u}Error Analysis in $u$}

In the following, we will discuss how to derive bounds on the rate
of convergence of $u$ instead of $U$. As a basis for both proofs
of this section, we start with the discrete FP operator which we rewrite
as 
\[
-\sum_{j:\,j\sim i}\frac{m_{ij}}{h_{ij}}\kappa_{ij}S_{ij}\left(\frac{u_{j}}{\pi_{j}}-\frac{u_{i}}{\pi_{i}}\right)=-\sum_{j:\,j\sim i}\frac{m_{ij}}{h_{ij}}\kappa_{ij}\left(u_{j}-u_{i}\right)-\sum_{j:\,j\sim i}\frac{m_{ij}}{h_{ij}}\kappa_{ij}\left(\frac{S_{ij}-\pi_{j}}{\pi_{j}}u_{j}-\frac{S_{ij}-\pi_{i}}{\pi_{i}}u_{i}\right)\,.
\]
We have
\[
\mathfrak{E}_{\cT,{\rm FPE}}\left(U;v\right)-\mathfrak{E}_{\cT}\left(u;v\right)=\sum_{i\sim j}\left(\frac{m_{ij}}{h_{ij}}\kappa_{ij}\left(\frac{S_{ij}-\pi_{j}}{\pi_{j}}u_{j}-\frac{S_{ij}-\pi_{i}}{\pi_{i}}u_{i}\right)-\int_{\sigma_{ij}}\kappa u\nabla V\cdot\bnu_{ij}\right)\left(v_{j}-v_{i}\right),
\]
where we want to estimate the right-hand side. For $V_{i}-V_{j}=O\left(h\right)$
we have
\begin{equation}
\frac{S_{ij}-\pi_{j}}{\pi_{j}}=\frac{1}{2}\left(\frac{\pi_{i}}{\pi_{j}}-1\right)+O\left(\pi_{i}-\pi_{j}\right)=\frac{1}{2}\left(V_{j}-V_{i}\right)+O\left(\pi_{i}-\pi_{j}\right)+O\left(V_{i}-V_{j}\right)\label{eq:ErrorEstimateu}
\end{equation}
and hence
\[
\mathfrak{E}_{\cT,{\rm FPE}}\left(U;v\right)-\mathfrak{E}_{\cT}\left(u;v\right)=\sum_{i\sim j}\left(\frac{m_{ij}}{h_{ij}}\kappa_{ij}\frac{1}{2}\left(V_{j}-V_{i}\right)\left(u_{i}+u_{j}\right)-\int_{\sigma_{ij}}\kappa u\nabla V\cdot\bnu_{ij}+O(h)\right)\left(v_{j}-v_{i}\right)\,.
\]
Since $\kappa_{ij}\approx\kappa$., $\frac{u_{i}+u_{j}}{2}\approx u$,
$\frac{V_{j}-V_{i}}{h_{ij}}\approx\nabla V$ it holds $\mathfrak{E}_{\cT,{\rm FPE}}\left(U;v\right)\approx\mathfrak{E}_{\cT}\left(u;v\right)$.
\begin{thm}
For smooth potentials $V\in C^{2}$ it holds $\|\mathfrak{E}_{\cT}\left(u;v\right)\|_{H_{\cT,\kappa S}^{*}}=O(h).$
\end{thm}

\begin{rem}
\label{rem:SG-strong-grad}The calculation \eqref{eq:ErrorEstimateu}
is an approximation for small values of $\left|V_{j}-V_{i}\right|$.
In the particular case of large discrete gradients a general approximation
of $\frac{S_{ij}-\pi_{j}}{\pi_{j}}$ is not at hand. However, in the
SG case $S_{\ast}=S_{0,-1}$ we observe (compare with \eqref{eq:limit-B-SG}
and \eqref{eq:B-vs-S-SG}) introducing $f\left(x\right)=\frac{-x-e^{x}-1}{\left(e^{x}-1\right)x}$
(with $f\left(x\right)\to0$ as $x\to+\infty$ and $f\left(x\right)\to1$
as $x\to-\infty$)
\begin{align*}
\frac{1}{h_{ij}}\frac{S_{ij}-\pi_{j}}{\pi_{j}} & =\frac{1}{h_{ij}}\frac{V_{j}-V_{i}-\left(\e^{V_{i}-V_{j}}-1\right)}{\e^{V_{i}-V_{j}}-1}\\
 & =\frac{V_{i}-V_{j}}{h_{ij}}f\left(V_{i}-V_{j}\right)\to\begin{cases}
-\nabla V\cdot\bnu_{ij} & \text{if }V_{i}\gg V_{j}\\
0 & \text{if }V_{j}\gg V_{i}
\end{cases}\quad\text{as }h_{ij}\to0\,.
\end{align*}
Hence we observe that the SG method is particularly suited to minimize
the error term 
\[
\frac{m_{ij}}{h_{ij}}\kappa_{ij}\left(\frac{S_{ij}-\pi_{j}}{\pi_{j}}u_{j}-\frac{S_{ij}-\pi_{i}}{\pi_{i}}u_{i}\right)-\int_{\sigma_{ij}}\kappa u\nabla V\cdot\bnu_{ij}
\]
 for large gradients $\nabla V$.
\end{rem}

\subsection{\label{subsec:Qualitative-comparison-on-CM}Qualitative comparison
on cubic meshes}

In view of Section \ref{subsec:Consistency-on-cubic} we consider
a polygonal domain $\Omega\subset\R^{d}$ with $d\leq3$ and a cubic
mesh where $\Omega_{i}=x_{i}+[-h/2,h/2]^{d}$, $x_{i}\in h\mathbb{Z}\subset\Omega$
to show that $\left|\int_{\sigma_{ij}}\left(\pi-S_{ij}\right)\kappa\nabla U\cdot\bnu_{ij}\right|=O(h^{2}).$
In fact the following calculations are quite standard and, therefore,
we shorten our considerations. We have for $x\in\sigma_{ij}$

\begin{align*}
S_{ij}-\pi(x) & =S(\pi_{i},\pi_{j})-S(\pi(x),\pi(x))=\\
 & =\nabla S(x)\cdot\begin{pmatrix}\pi_{i}-\pi(x)\\
\pi_{j}-\pi(x)
\end{pmatrix}+\begin{pmatrix}\pi_{i}-\pi(x)\\
\pi_{j}-\pi(x)
\end{pmatrix}\cdot\nabla^{2}S(x)\cdot\begin{pmatrix}\pi_{i}-\pi(x)\\
\pi_{j}-\pi(x)
\end{pmatrix}+O(h^{3}).
\end{align*}

Moreover, we have $\pi_{i}-\pi(x)=\nabla\pi\cdot(x_{i}-x)$. The gradient
of $S$ is given by $(1/2,1/2)^{T}$ and hence, we $S_{ij}-\pi(x)=\frac{\pi_{i}+\pi_{j}-2\pi(x)}{2}+O(h^{2}).$
We compute the first term in more detail. We have $\pi_{j}-\pi(x)=\nabla\pi\cdot(x_{j}-x)$
and $\pi_{i}-\pi(x)=\nabla\pi\cdot(x_{i}-x)$ and the sum yields $\pi_{i}+\pi_{j}-2\pi(x)=\nabla\pi\cdot(x_{i}+x_{j}-2x)=\frac{1}{2}\nabla\pi\cdot\tilde{x}$,
where $\tilde{x}=x-\frac{x_{i}+x_{j}}{2}$ the coordinate on the cell
surface with respect to the middle point $\bar{x}=\frac{x_{i}+x_{j}}{2}$.
Hence, we get
\[
\int_{\sigma_{ij}}(\pi-S_{ij})\kappa\nabla U\cdot\nu_{ij}=\frac{1}{4}\int_{\sigma_{ij}}\nabla\pi(x)\cdot\tilde{x}\kappa(x)\nabla U(x)\cdot\nu_{ij}\d\sigma(\tilde{x})+O(h^{2}).
\]

Now we can fix the function $s(x)=\kappa(x)\nabla U(x)\cdot\nu_{ij}\nabla\pi(x)$
with respect to $\bar{x}$. We have $s(x)=s(\bar{x})+(x-\bar{x})\nabla s(\bar{x})+O(h^{2})$,
which implies (assuming that $U,\pi\in\mathrm{C}^{2}$ and $\kappa\in\mathrm{C}^{1}$)
that $\int_{\sigma_{ij}}(\pi-S_{ij})\kappa\nabla U\cdot\nu_{ij}=\frac{1}{4}\int_{\sigma_{ij}}\left(s(\bar{x})+(x-\bar{x})\nabla s(\bar{x})\right)\cdot\tilde{x}\d\sigma(\tilde{x})+O(h^{2})=\frac{1}{4}\int_{\sigma_{ij}}s(\bar{x})\cdot\tilde{x}\d\sigma(\tilde{x})+O(h^{2})$.
But the first vanishes, since the interface $\sigma_{ij}$ is symmetric
w.r.t. the mid point $\bar{x}$ and we are integrating along $\tilde{x}$.
Hence, we have $\left|\int_{\sigma_{ij}}\left(\pi-S_{ij}\right)\kappa\nabla U\cdot\bnu_{ij}\right|=O(h^{2}).$

Hence, iterating the above argument twice for $\kappa$ and $\pi$
and exploiting in the first step Theorem \ref{thm:ConsistencyCubicGrid}
we proved the following.
\begin{thm}
Let $d\leq3$. On a polygonal domain $\Omega\subset\R^{d}$ with a
cubic mesh where $\Omega_{i}=x_{i}+[-h/2,h/2]^{d}$, $x_{i}\in h\mathbb{Z}\subset\Omega$,
it holds 
\[
\norm{\mathfrak{E}_{\cT,{\rm FPE}}\left(U;\cdot\right)}_{H_{\cT,\kappa S}^{\ast}}\leq Ch^{2}.
\]
\end{thm}

\section{\label{sec:Interpretation-and-simulation}Numerical simulation and
convergence analysis}

In this section, we provide a numerical convergence analysis of the
flux discretization schemes based on weighted Stolarsky means described in
the previous sections. For the sake of simplicity, we restrict ourselves
to one-dimensional examples, for which already non-trivial results
can be observed.
\begin{example}
\label{exa: example 1}We consider the potential $V\left(x\right)=2\sin{\left(2\pi x\right)}$
and the right hand side $f\left(x\right)=x\left(1-x\right)$ on $x=\left(0,1\right)$.
We assume the  diffusion constant $\kappa=1$ and Dirichlet boundary conditions $u\left(0\right)=0$
and $u\left(1\right)=1$. The Stolarsky mean discretizations are compared
point-wise with a numerically computed reference solution $u_{\text{ref}}$
(and $J_{\text{ref}}$) that was obtained by the shooting method
(using a fourth order Runge--Kutta scheme) in combination with Brent's
root finding algorithm \cite{Brent1971} on a very fine grid with
$136474$ nodes ($h\approx 7.3\times 10^{-6}$).

\begin{figure}[t]
\includegraphics[width=1\textwidth]{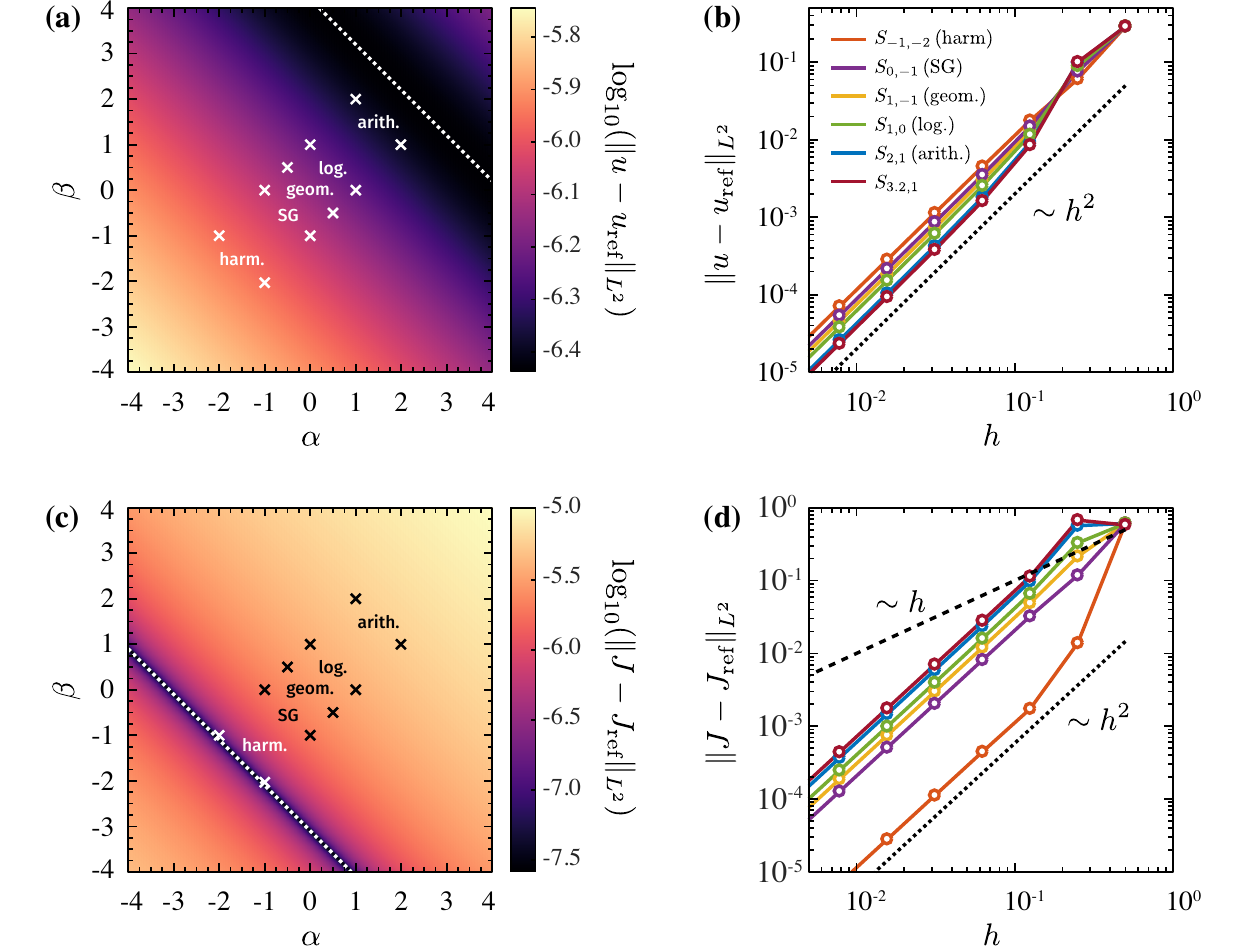}\caption{Numerical results for Example \ref{exa: example 1}. \textbf{(a)}~Discretization
error $\log_{10}(\Vert u-u_{\text{ref}}\Vert_{L_{2}})$ in the $\left(\alpha,\beta\right)$-plane
on an equidistant mesh with $2^{10}+1$ nodes. The error is color-coded.
Several special means (see Tab.\,\ref{tab: stolarsky means}) are highlighted by crosses. Notice the symmetry
$S_{\alpha,\beta}\left(x,y\right)=S_{\beta,\alpha}\left(x,y\right)$.
\textbf{(b)}~Quadratic convergence of the discrete solution to the
exact reference solution $u_{\text{exact}}$ under mesh refinement
in the $L_{2}$-norm. See the inset for a legend and color-coding
of the considered means $S_{\alpha,\beta}$. In the present example,
the best numerical result for $u$ is achieved by $S_{3.2,1}$. \textbf{(c)}~Logarithmic
error of the numerically computed flux density $\log_{10}(\Vert J-J_{\text{ref}}\Vert_{L_{2}})$
in the $\left(\alpha,\beta\right)$-plane on the same mesh as in (a).
\textbf{(d)}~Convergence of the numerically computed flux density
to $J_{\text{ref}}$. In contrast to the convergence of $u$ shown
in (b), here the harmonic average $S_{-1,-2}$ yield the highest accuracy.}
\label{fig: example 1}
\end{figure}
\end{example}

The convergence results are summarized in Fig.~\ref{fig: example 1}.
In Figure~\ref{fig: example 1}\,(a), the logarithmic error $\log_{10}(\Vert u-u_{\text{ref}}\Vert_{L_{2}})$
is shown in the $\left(\alpha,\beta\right)$-plane of the Stolarsky
mean parameters for an equidistant mesh with $2^{10}+1=1025$ nodes.
First, we note that the accuracy for a mean $S_{\alpha,\beta}$ is
indeed practically invariant along $\alpha+\beta=\mathrm{const}$,
which is consistent with our analytical result in Section~\ref{sec:Comparison-of-discretization}.
In this particular example, we observe optimal accuracy at about $\alpha+\beta\approx4.2$.
This coincides with the convergence results under mesh refinement
shown in Figure~\ref{fig: example 1}\,(b), where the fastest convergence
is obtained for the scheme involving the $S_{3.2,1}$-mean. The other
considered schemes, however, show as well a quadratic convergence
behavior with a slightly larger constant. Interestingly, for the same
example, we find that the optimal mean for an accurate approximation
of the flux $J$ is on $\alpha+\beta=-3$, see Figure~\ref{fig: example 1}\,(c).
This is further evidences in Figure~\ref{fig: example 1}\,(d),
where the harmonic mean $S_{-1,-2}$ converges significantly faster
than the other schemes. Obviously, in the present example, the minimal
attainable error for both $u$ and $J$ can not be achieved by the
same discretization scheme.

\begin{example}
\label{exa: example 2}We consider the potential $V\left(x\right)=5\left(x+1\right)x$.
The right hand side function, the diffusion constant and the boundary conditions
are the same as in Example \ref{exa: example 1}. The problem has an exact solution
involving the imaginary error function (which is related to the Dawson
function), that has been obtained using Wolfram Mathematica \cite{Mathematica2017}.

\begin{figure}[t]
\includegraphics[width=1\textwidth]{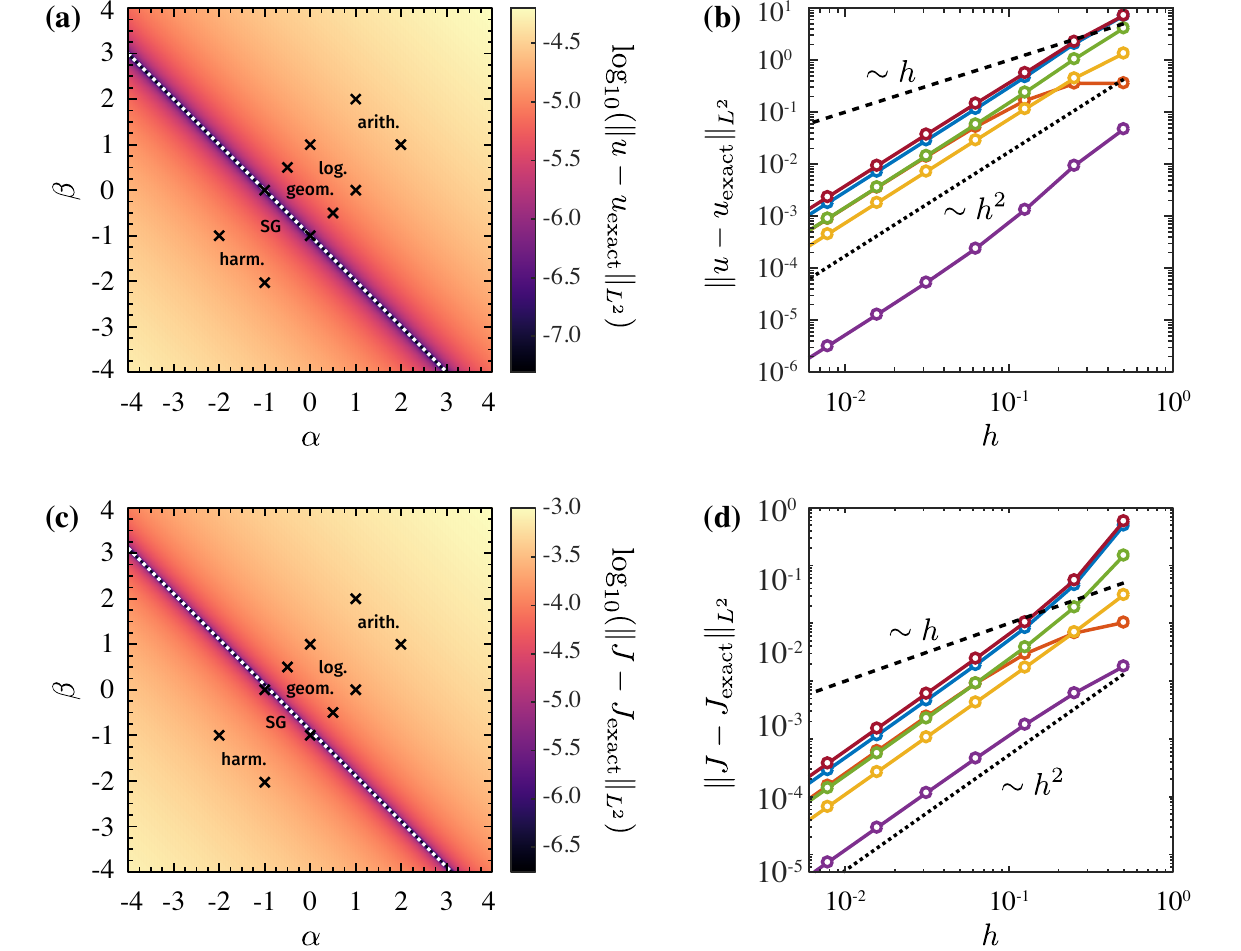}\caption{Discretization errors and convergence behavior of the numerically computed $u$ and $J$ in Example \ref{exa: example 2} using the
Stolarsky mean schemes. The errors in \textbf{(a)} and \textbf{(c)} are color-coded. The coloring of the means in \textbf{(b)} and \textbf{(d)}
is the same as in Figure~\ref{fig: example 1}\,(b). The plots clearly
show a superior performance of the Scharfetter--Gummel scheme, which
corresponds to the Stolarsky mean $S_{0,-1}$ for the approximation
of both the density $u$ and the flux $J$.}

\label{fig: example 2}
\end{figure}

The numerical results are show in Figure~\ref{fig: example 2}. The
discretization errors of both the density $u$ and the flux $J$
shown in Figure~\ref{fig: example 2}\,(a) and (c) exhibit a sharp
minimum on $\alpha+\beta=-1$. This includes the Scharfetter--Gummel
mean $S_{0,-1}$, which converges fastest to the exact reference solutions
for $u$ and $J$, as shown in \ref{fig: example 2}\,(b) and (d).
The SQRA scheme, with geometric mean $S_{\alpha,-\alpha}$, is found
to be second best in the present example.
\end{example}

The numerical results are in line with our previous statements from
Remark \ref{rem:SG-strong-grad}: In the case of strong gradients
$\nabla V$, the Scharfetter--Gummel scheme provides the most accurate
flux discretization, in particular, the SG mean $S_{0,-1}$ is the
only Stolarsky mean that recovers the upwind scheme \eqref{eq:limit-B-SG}.
Away from that drift-dominated regime, the situation is less clear
and other averages $S_{\alpha,\beta}$ can be superior, see for instance
Example \ref{exa: example 1}.

\appendix

\section{Appendix}

\subsection{A General Poincaré Inequality}

We derive a general Poincaré inequality on meshes. The idea behind
the proof seems to go back to Hummel \cite{Hummel1999} and has been
adapted in a series of works e.g. \cite{heida2018convergences,heida2017fractal}.
Let $e_{0}=0$ and $(e_{i})_{i=1,\dots,n}$ be the canonical basis
of $\Rn$. Define: 
\[
D^{d-1}:=\{\nu\in\S^{d-1}\,\,|\,\,\exists m\in\left\{ 1,\cdots,d\right\} :\nu\cdot e_{i}=0\,\,\forall\,\,i\in\left\{ 0,1,\cdots,m-1\right\} \,\,\textnormal{and}\,\,\nu\cdot e_{m}>0\}\,.
\]
Every $\nu\in\S^{d-1}$ satisfies $\nu\cdot e_{i}\neq0$ for at least
one $e_{i}$. Thus, for every $\nu\in\S^{d-1}$ it holds $\nu\in D^{d-1}$
if and only if $-\nu\not\in D^{d-1}$. 

We denote $\Gamma=\bigcup_{\sigma\in\cE_{\Omega}}\sigma$ and say
that $x\in\Gamma$ is a Lipschitz point if $\Gamma$ is a Lipschitz
graph in a neighborhood of $x$. The set of Lipschitz-Points is called
$\Gamma_{L}\subset\Gamma$ and we note that for the $\left(d-1\right)$-dimensional
Hausdorff-measure of $\Gamma\backslash\Gamma_{L}$ it holds $\mathcal{H}^{d-1}\left(\Gamma\backslash\Gamma_{L}\right)=0$.

For $x\in\Gamma_{L}$, we denote $\nu_{x}\in D^{d-1}$ the normal
vector to $\Gamma$ in $x$.. Let 
\[
\mathcal{C}_{0}^{1}(\Omega;\Gamma):=\left\{ u\in C(\Omega\backslash\Gamma)\,\,:\;\,u|_{\partial\Omega}\equiv0\,,\;\forall i\,\exists v_{i}\in C^{1}\left(\overline{\Omega_{i}}\right):\,u|_{\Omega_{i}}=v_{i}\right\} 
\]
and for $u\in\mathcal{C}_{K,0}^{1}(\Omega)$ define in Lipschitz points
$x\in\Gamma_{L}$ 
\[
u_{\pm}(x):=\lim_{h\to0}\left(u\left(x\pm h\nu_{x}\right)\right)\,,~~~\jump u(x):=u_{+}(x)-u_{-}(x)\,.
\]
For two points $x,y\in\Rn$ denote $(x,y)$ the closed straight line
segment connecting $x$ and $y$ and for $\xi\in(x,y)\cap\Gamma_{L}$
denote 
\[
\jump u_{x,y}(\xi):=\lim_{h\to0}\left(u\left(\xi+h(y-x)\right)-u\left(\xi-h(y-x)\right)\right)
\]
the jump of the function $u$ at $\xi$ in direction $(y-x)$, i.e.
$\jump u_{x,y}(\xi)\in\pm\jump u\left(\xi\right)$. We can extend
$\jump u$ to $\Gamma$ by $\jump u\left(x\right)=0$ for $x\in\Gamma\backslash\Gamma_{L}$
and define 
\begin{align*}
\norm u_{H^{1}\left(\Omega;\Gamma\right)} & :=\left(\int_{\Omega\backslash\Gamma}\left|\nabla u\right|^{2}+\int_{\Gamma}\jump u^{2}\right)^{\frac{1}{2}}\,,\\
H_{0}^{1}\left(\Omega;\Gamma\right) & :=\overline{\mathcal{C}_{0}^{1}\left(\Omega;\Gamma\right)}^{\norm{\cdot}_{H^{1}\left(\Omega;\Gamma\right)}}\,.
\end{align*}
Then we find the following result:
\begin{lem}[Semi-discrete Poincaré inequality]
\label{lem:Poincare}Let $\Omega\subset\Rd$ be a bounded domain.
The space $H_{0}^{1}\left(\Omega;\Gamma\right)$ is linear and closed
for every $s\in[0,\frac{1}{2})$ and there exists a positive constant
$C_{s}>0$ such that the following holds: Suppose there exists a constant
$C_{\#}>0$ such that for almost all $\left(x,y\right)\in\Omega^{2}$
it holds $\#\left(\left(x,y\right)\cap\Gamma\right)\leq C_{\#}$..
Then for every $u\in H_{0}^{1}\left(\Omega;\Gamma\right)$ it holds
\begin{equation}
\left\Vert u\right\Vert _{H^{s}(\Omega)}^{2}\leq C_{s}\left(C_{\#}\int_{\Gamma}\jump u^{2}+\left\Vert \nabla u\right\Vert _{L^{2}(\Omega\backslash\Gamma)}^{2}\right)\,.\label{eq:Poincare-H-1-2}
\end{equation}
Furthermore, for every $u\in H^{1}\left(\Omega;\Gamma\right)$ and
every $\boldsymbol{\eta}\in\R^{d}$ it holds
\begin{equation}
\int_{\Omega}\left|u(x)-u(x+\boldsymbol{\eta})\right|^{2}dx\leq\left|\boldsymbol{\eta}\right|\left(C_{\#}\int_{\Gamma}\jump u^{2}+\left\Vert \nabla u\right\Vert _{L^{2}(\Omega\backslash\Gamma)}^{2}\right)\,.\label{eq:Poincare-L-2}
\end{equation}
\end{lem}

\begin{proof}
In what follows, given $u\in\mathcal{C}_{0}^{1}(\Omega;\Gamma)$,
we write $\widehat{\nabla u}(x):=\nabla u(x)$ if $x\in\Omega\backslash\Gamma$
and $\widehat{\nabla u}(x)=0$ else. For $y\in\Rd$ we denote $\left(x,y\right)=\left\{ x+s\left(y-x\right)\,:\;s\in[0,1]\right\} $.
Using $2ab<a^{2}+b^{2}$, we infer for $u\in\mathcal{C}_{0}^{1}(\Omega;\Gamma)$
and $x,y\in\overline{\Omega}\backslash\Gamma$ such that $\left(x,y\right)\cap\Gamma$
is finite the inequality
\begin{align*}
\left|u(x)-u(y)\right|^{2} & \leq\left(\sum_{\xi\in(x,y)\cap\Gamma}\jump u_{x,y}(\xi)+\int_{0}^{1}\widehat{\nabla u}\left(x+s(y-x)\right)\cdot\left(x-y\right)ds\right)^{2}\\
 & <\left|x-y\right|^{2}\int_{0}^{1}\left|\widehat{\nabla u}\left(x+s(y-x)\right)\right|^{2}ds+\left(\sum_{\xi\in(x,y)\cap\Gamma}\jump u_{x,y}(\xi)\right)^{2}
\end{align*}
Since $\jump u_{x,y}=\jump u$ we compute 
\[
\left(\sum_{\xi\in(x,y)\cap\Gamma}\jump u_{x,y}(\xi)\right)^{2}\leq\#\left(\left(x,y\right)\cap\Gamma\right)\sum_{\xi\in(x,y)\cap\Gamma}\jump u^{2}(\xi)
\]
and obtain 
\begin{align}
\left|u(x)-u(y)\right|^{2} & <\left|x-y\right|^{2}\int_{0}^{1}\left|\widehat{\nabla u}\left(x+s(y-x)\right)\right|^{2}ds\nonumber \\
 & \quad+\#\left(\left(x,y\right)\cap\Gamma\right)\sum_{\xi\in(x,y)\cap\Gamma}\jump u^{2}(\xi)\,.\label{eq:fundamental-estimate}
\end{align}
We fix $\eta>0$ and consider the orthonormal basis $(e_{i})_{i=1,\dots,d}$
of $\Rd$. The determinant of the first fundamental form of $\Gamma$
is bigger than $1$ almost everywhere. Hence we can observe that 
\begin{align*}
\int_{\Omega}\sum_{\xi\in(x,x+\eta e_{1})\cap\Gamma}\jump u^{2}(\xi)\,\d x & =\int_{\R}\left(\int_{\R^{d-1}}\sum_{\xi\in(x,x+\eta e_{1})\cap\Gamma}\jump u^{2}(\xi)\,\d x_{2}\dots\d x_{d}\right)\d x_{1}\\
 & \leq\int_{\R}\int_{\Gamma\cap\left(\left(x_{1},x_{1}+\eta\right)\times\R^{d-1}\right)}\jump u^{2}(x)\,\d\sigma\,\d x_{1}\\
 & \leq\eta\int_{\Gamma}\jump u^{2}(x)\,\d x\,,
\end{align*}
where we used that the surface elements are bigger than $1$. Furthermore,
we have 
\[
\eta^{2}\int_{0}^{1}\left|\widehat{\nabla u}\left(x+s\eta e_{1}\right)\right|^{2}ds=\eta\int_{0}^{\eta}\left|\widehat{\nabla u}\left(x+se_{1}\right)\right|^{2}ds\,.
\]
Replacing $e_{1}$ in the above calculations with any unit vector
$e$, we obtain from integration of \eqref{eq:fundamental-estimate}
with $y=x+\boldsymbol{\eta}$, $\boldsymbol{\eta}=\eta e$, over $\Omega$
that 
\[
\int_{\Omega}\left|u(x)-u(x+\boldsymbol{\eta})\right|^{2}dx\leq\left|\boldsymbol{\eta}\right|\left(C_{\#}\int_{\Gamma}\jump u^{2}+\left\Vert \nabla u\right\Vert _{L^{2}(\Omega\backslash\Gamma)}^{2}\right)\,.
\]
Dividing by $\left|\boldsymbol{\eta}\right|$ and integrating over
$\boldsymbol{\eta}\in\Rd$, we obtain that for every $s\in[0,\frac{1}{2})$
there exists a positive constant $C_{s}>0$ independent from $u$
and $K$ such that 
\begin{equation}
\left\Vert u\right\Vert _{H^{s}(\Omega)}^{2}\leq C_{s}\left(C_{\#}\int_{\Gamma}\jump u^{2}+\left\Vert \nabla u\right\Vert _{L^{2}(\Omega\backslash\Gamma)}^{2}\right)\,.\label{eq:H-12-estimate-finite-K}
\end{equation}
Hence, by approximation, the last two estimates hold for all $u\in H_{0}^{1}\left(\Omega;\Gamma\right)$..
\end{proof}

\subsection{\label{subsec:Proof-of-Theorem-Stolarsky-grad}Physical relevance
of the geometric mean}
\begin{thm}
\label{thm:Stolarsky-symmetry}Let $S_{ij}=S_{\ast}\left(\pi_{i},\pi_{j}\right)$
be a Stolarsky mean and let $\mathsf{\psi^{*}}$ be a symmetric strictly
convex function with $\mathsf{\psi^{*}}(0)=0$. If $\partial_{\pi}\left(S_{ij}a_{ij}\right)=0$
then $S_{ij}=\sqrt{\pi_{i}\pi_{j}}$ and $\mathsf{\psi^{*}}$ is proportional
to $\mathsf{C}^{*}$.
\end{thm}

\begin{proof}[Proof of Theorem \ref{thm:Stolarsky-symmetry}]
The case $S_{ij}=\sqrt{\pi_{i}\pi_{j}}$ and $\mathsf{\psi^{*}}(\xi)=\cosh\xi-1$
was explained in detail in \cite{heida2018convergences}.

In the general case, symmetry of $\mathsf{\psi^{*}}$ in $\xi_{i}-\xi_{j}$
implies $\mathsf{\psi^{*}}\left(\xi_{i}-\xi_{j}\right)=\mathsf{\psi^{*}}\left(\left|\xi_{i}-\xi_{j}\right|\right)$.
We make use of the fact that the original $\mathsf{C^{*}}(\xi)=\cosh\xi-1$
is a bijection on $[0,\infty)$ and suppose that hence $\mathsf{\psi^{*}}\left(\xi_{i}-\xi_{j}\right)=\theta\left(\mathsf{C^{*}}\left(\xi_{i}-\xi_{j}\right)\right)$.
This implies particularly that 
\[
0\leq x\,\partial_{x}\left(\theta\left(\mathsf{C^{*}}(x)\right)\right)=x\,\partial_{\xi}\theta\left(\mathsf{C^{*}}(x)\right)\partial_{x}\mathsf{C^{*}}(x)\,.
\]
Furhtermore, the symmetry of $\mathsf{\psi^{*}}$ implies by the last
inequality that $\partial_{\xi}\theta\left(\mathsf{C^{*}}(x)\right)>0$.
Inserting this information in \eqref{eq:S_ij--general-Psi} and \eqref{eq:S_ij--general-a}
we observe that 
\[
S_{ij}\left(\frac{u_{i}}{\pi_{i}}-\frac{u_{j}}{\pi_{j}}\right)\partial_{\xi}\theta\left(\mathsf{C^{*}}\left(\ln\left(\frac{u_{i}}{\pi_{i}}\right)-\ln\left(\frac{u_{j}}{\pi_{j}}\right)\right)\right)^{-1}\sinh\left(\ln\left(\frac{u_{i}}{\pi_{i}}\right)-\ln\left(\frac{u_{j}}{\pi_{j}}\right)\right)^{-1}
\]
has to be independent from $\pi_{i}$ and $\pi_{j}$. From the above
case $S_{ij}=\sqrt{\pi_{i}\pi_{j}}$, we know that 
\[
\sqrt{\pi_{i}\pi_{j}}\left(\frac{u_{i}}{\pi_{i}}-\frac{u_{j}}{\pi_{j}}\right)\sinh\left(\ln\left(\frac{u_{i}}{\pi_{i}}\right)-\ln\left(\frac{u_{j}}{\pi_{j}}\right)\right)^{-1}
\]
is constant in $\pi_{i}$ and $\pi_{j}$. Hence it remains to show
that 
\[
f\left(\pi_{i},\pi_{j}\right):=S_{ij}\sqrt{\pi_{i}\pi_{j}}^{-1}\partial_{\xi}\psi\left(\frac{u_{i}}{u_{j}}\frac{\pi_{j}}{\pi_{i}}+\frac{u_{j}}{u_{i}}\frac{\pi_{i}}{\pi_{j}}\right)^{-1}
\]
is independent from $\pi_{i}$ and $\pi_{j}$ if and only if $\partial_{\xi}\psi=\mathrm{const}$
and $S_{ij}=\sqrt{\pi_{i}\pi_{j}}$.

Assume first that $S_{ij}\sqrt{\pi_{i}\pi_{j}}^{-1}=\mathrm{const}$. Then
for $p=\frac{\pi_{i}}{\pi_{j}}$ we obtain that 
\[
\partial_{p}\left(\partial_{\xi}\theta\left(\frac{u_{i}}{u_{j}}p^{-1}+\frac{u_{j}}{u_{i}}p\right)^{-1}\right)=0
\]
has to hold. This implies that $\partial_{\xi}\psi=\mathrm{const}$.

If $S_{ij}\sqrt{\pi_{i}\pi_{j}}^{-1}\not=\mathrm{const}$, we use the definition
of the weighted Stolarsky means given in \eqref{eq:Stolarsky-means}
and note that 
\[
S_{ij}:=S\left(\pi_{i},\pi_{j}\right)=\left(\frac{\beta(\pi_{i}^{\alpha}-\pi_{j}^{\alpha})}{\alpha(\pi_{i}^{\beta}-\pi_{j}^{\beta})}\right)^{\frac{1}{\alpha-\beta}}=\pi_{j}\left(\frac{\beta(p^{\alpha}-1)}{\alpha(p^{\beta}-1)}\right)^{\frac{1}{\alpha-\beta}}\,,
\]
where again $p=\frac{\pi_{i}}{\pi_{j}}$. Hence we obtain that
\begin{align*}
f\left(\pi_{i},\pi_{j}\right) & =\tilde{f}(p):=\sqrt{\frac{1}{p}}\left(\frac{\beta(p^{\alpha}-1)}{\alpha(p^{\beta}-1)}\right)^{\frac{1}{\alpha-\beta}}\partial_{\xi}\theta\left(\frac{u_{i}}{u_{j}}p^{-1}+\frac{u_{j}}{u_{i}}p\right)^{-1}\\
 & =\left(\frac{\beta\left(p^{\frac{\alpha}{2}}-p^{-\frac{\alpha}{2}}\right)}{\alpha\left(p^{\frac{\beta}{2}}-p^{-\frac{\beta}{2}}\right)}\right)^{\frac{1}{\alpha-\beta}}\partial_{\xi}\theta\left(\frac{u_{i}}{u_{j}}p^{-1}+\frac{u_{j}}{u_{i}}p\right)^{-1}
\end{align*}
has to be independent of $\pi_{i}$ and $\pi_{j}$. But then, $\tilde{f}$
is independent of $p$. Now, we define $a=\frac{u_{j}}{u_{i}}$ and
observe that 
\[
\tilde{f}\left(\frac{1}{a^{2}p}\right)=\left(\frac{\beta\left(\left(a^{2}p\right)^{-\frac{\alpha}{2}}-\left(a^{2}p\right)^{\frac{\alpha}{2}}\right)}{\alpha\left(\left(a^{2}p\right)^{-\frac{\beta}{2}}-\left(a^{2}p\right)^{\frac{\beta}{2}}\right)}\right)^{\frac{1}{\alpha-\beta}}\partial_{\xi}\theta\left(\frac{u_{i}}{u_{j}}p^{-1}+\frac{u_{j}}{u_{i}}p\right)^{-1}.
\]
We assume for $\alpha\neq\beta$. The case $\alpha=\beta$ can follows
by continuity. For any $p$ it should holds $\tilde{f}\left(\frac{1}{a^{2}p}\right)=\tilde{f}(p)$,
which implies 
\[
\left(\frac{\beta\left(p^{\frac{\alpha}{2}}-p^{-\frac{\alpha}{2}}\right)}{\alpha\left(p^{\frac{\beta}{2}}-p^{-\frac{\beta}{2}}\right)}\right)^{\frac{1}{\alpha-\beta}}=\left(\frac{\beta\left(\left(a^{2}p\right)^{-\frac{\alpha}{2}}-\left(a^{2}p\right)^{\frac{\alpha}{2}}\right)}{\alpha\left(\left(a^{2}p\right)^{-\frac{\beta}{2}}-\left(a^{2}p\right)^{\frac{\beta}{2}}\right)}\right)^{\frac{1}{\alpha-\beta}},
\]
or equivalently, after introducing $q^{2}=p$,
\[
\left(a^{\alpha}-a^{\beta}\right)q^{\alpha+\beta}+\left(a^{\beta}-a^{-\alpha}\right)q^{\beta-\alpha}+\left(a^{-\beta}-a^{\alpha}\right)q^{\alpha-\beta}+\left(a^{-\alpha}-a^{-\beta}\right)q^{-\beta-\alpha}=0.
\]

Since $\alpha\neq\beta$, one of the terms $q^{\pm\alpha\pm\beta}$
grows faster than the other. Hence we conclude that $a^{\alpha}=a^{\pm\beta}$
which means, $a=1$, a contradiction.
\end{proof}

\subsection{\label{sec:Verification-of-eq:Stolarsky-second-deriv}Properties
of the Stolarsky mean}
\begin{lem}
\label{lem:Stolarsky-derivative}For every of the above Stolarsky
means $S_{\ast}(x,y)$ it holds
\[
\partial_{x}S_{\ast}(x,x)=\partial_{y}S_{\ast}\left(x,x\right)=\frac{1}{2}~~\text{and}~~\partial_{x}^{2}S_{\ast}\left(x,x\right)=\partial_{y}^{2}S_{\ast}\left(x,x\right)=-\partial_{xy}^{2}S_{\ast}\left(x,x\right)=-\partial_{yx}^{2}S_{\ast}\left(x,x\right)\,.
\]
\end{lem}

\begin{proof}
Since $S_{\ast}(x,x)=x$ and $S_{\ast}$ is symmetric in $x$ and
$y$, we find from differentiating $\partial_{x}S_{\ast}=\partial_{y}S_{\ast}=\frac{1}{2}$.
From the last equality, we find $\partial_{x}S_{\ast}(x,x)-\partial_{y}S_{\ast}(x,x)=0$
as well as $\partial_{x}S_{\ast}(x,x)+\partial_{y}S_{\ast}(x,x)=1$
and differentiation yields 
\begin{align}
\partial_{x}^{2}S_{\ast}\left(x,x\right)-\partial_{y}^{2}S_{\ast}\left(x,x\right)-\partial_{xy}^{2}S_{\ast}\left(x,x\right)+\partial_{yx}^{2}S_{\ast}\left(x,x\right) & =0\,,\label{eq:Stolarsky-help-1-1}\\
\partial_{x}^{2}S_{\ast}\left(x,x\right)+\partial_{y}^{2}S_{\ast}\left(x,x\right)+\partial_{xy}^{2}S_{\ast}\left(x,x\right)+\partial_{yx}^{2}S_{\ast}\left(x,x\right) & =0\,.\label{eq:Stolarsky-help-2-1}
\end{align}
Since $-\partial_{xy}^{2}S_{\ast}\left(x,x\right)+\partial_{yx}^{2}S_{\ast}\left(x,x\right)=0$,
equation \eqref{eq:Stolarsky-help-1-1} yields $\partial_{x}^{2}S_{\ast}\left(x,x\right)=\partial_{y}^{2}S_{\ast}\left(x,x\right)$.
Inserting the last two relations into \eqref{eq:Stolarsky-help-2-1}
yields $\partial_{xy}^{2}S_{\ast}\left(x,x\right)=\partial_{yx}^{2}S_{\ast}\left(x,x\right)=-\partial_{x}^{2}S_{\ast}\left(x,x\right)$.
\end{proof}
\begin{lem}
It holds \eqref{eq:Stolarsky-second-deriv}$\partial_{x}^{2}S_{\alpha,\beta}\left(\pi,\pi\right)=\frac{1}{12\pi}\left(\alpha+\beta-3\right)$.
\end{lem}

\begin{proof}
We know from Lemma \ref{lem:Stolarsky-derivative} that $\partial_{x}S_{\alpha,\beta}\left(x,x\right)=\frac{1}{2}$
and $\partial_{x}^{2}S_{\alpha,\beta}\left(x,x\right)=-\partial_{y}\partial_{x}S_{\alpha,\beta}\left(x,x\right)$.
Hence we find 
\[
\partial_{x}S_{\alpha,\beta}\left(x+h,x-h\right)-\frac{1}{2}=\left(\begin{array}{c}
h\\
-h
\end{array}\right)\left(\begin{array}{c}
\partial_{x}^{2}S_{\alpha,\beta}\left(x,x\right)\\
\partial_{y}\partial_{x}S_{\alpha,\beta}\left(x,x\right)
\end{array}\right)=2h\partial_{x}^{2}S_{\alpha,\beta}\left(x,x\right)\,.
\]
We make use of the explicit form 
\[
\partial_{x}S_{\alpha,\beta}\left(x,y\right)=\left(\frac{\beta}{\alpha}\right)^{\frac{1}{\alpha-\beta}}\frac{\left(x^{\alpha}-y^{\alpha}\right)^{\frac{1}{\alpha-\beta}-1}}{\left(x^{\beta}-y^{\beta}\right)^{\frac{1}{\alpha-\beta}-1}}\frac{\alpha\left(x^{\beta}-y^{\beta}\right)x^{\alpha}-\beta\left(x^{\alpha}-y^{\alpha}\right)x^{\beta}}{\left(\alpha-\beta\right)\,\,x\,\,\left(x^{\beta}-y^{\beta}\right)^{2}}
\]
for $x\not=y$. We insert $x=x+h$ and $y=x-h$ and make use of the
following expansions
\begin{align*}
\left(\left(x+h\right)^{\alpha}-\left(x-h\right)^{\alpha}\right)^{c} & =\left(\alpha hx^{\alpha-1}\right)^{c}\left(2^{c}+O\left(h^{2}\right)\right)\\
\beta\left(\left(x+h\right)^{\alpha}-\left(x-h\right)^{\alpha}\right)\left(x+h\right)^{\beta} & =2\alpha\beta hx^{\alpha+\beta-1}+2\alpha\beta^{2}h^{2}x^{\alpha+\beta-2}\\
 & +\frac{1}{3}\alpha\beta h^{3}\left(\alpha^{2}-3\alpha+3\beta^{2}-3\beta+2\right)+O\left(h^{4}\right)\\
\alpha\left(\left(x+h\right)^{\beta}-\left(x-h\right)^{\beta}\right)\left(x+h\right)^{\alpha} & =2\alpha\beta hx^{\alpha+\beta-1}+2\alpha^{2}\beta h^{2}x^{\alpha+\beta-2}\\
 & +\frac{1}{3}\alpha\beta h^{3}\left(\beta^{2}-3\beta+3\alpha^{2}-3\alpha+2\right)+O\left(h^{4}\right)\\
\left(x+h\right)\,\left(\left(x+h\right)^{\beta}-\left(x-h\right)^{\beta}\right)^{2} & =4\beta^{2}h^{2}x^{2\beta-1}+4\beta^{2}h^{3}x^{2\beta-2}+O\left(h^{4}\right)
\end{align*}

\begin{multline*}
\alpha\left(\left(x+h\right)^{\beta}-\left(x-h\right)^{\beta}\right)\left(x+h\right)^{\alpha}-\beta\left(\left(x+h\right)^{\alpha}-\left(x-h\right)^{\alpha}\right)\left(x+h\right)^{\beta}\\
=2\alpha\beta\left(\alpha-\beta\right)h^{2}x^{\alpha+\beta-2}+\frac{\alpha\beta}{3}h^{3}x^{\alpha+\beta-3}\left(2\alpha^{2}-2\beta^{2}\right)+O\left(h^{4}\right)
\end{multline*}
to obtain 
\begin{align*}
\frac{\beta\left(x^{\alpha}-y^{\alpha}\right)x^{\beta}-\alpha\left(x^{\beta}-y^{\beta}\right)x^{\alpha}}{\left(\alpha-\beta\right)\,\,x\,\,\left(x^{\beta}-y^{\beta}\right)^{2}} & =\frac{\alpha\left(x^{\alpha+\beta-2}+h\frac{1}{3}x^{\alpha+\beta-3}\left(\alpha+\beta\right)+O\left(h^{2}\right)\right)}{2\beta\left(x^{2\beta-1}+hx^{2\beta-2}+O\left(h^{2}\right)\right)}
\end{align*}
and 
\[
\frac{\left(x^{\alpha}-y^{\alpha}\right)^{\frac{1}{\alpha-\beta}-1}}{\left(x^{\beta}-y^{\beta}\right)^{\frac{1}{\alpha-\beta}-1}}\approx\left(\frac{\alpha}{\beta}\right)^{\frac{1}{\alpha-\beta}-1}\left(\frac{x^{\alpha-1}\left(1+O\left(h^{2}\right)\right)}{x^{\beta-1}\left(1+O\left(h^{2}\right)\right)}\right)^{\frac{1}{\alpha-\beta}-1}\,.
\]
Together with 
\begin{align*}
\frac{a+bh}{c+dh} & =\frac{a}{c}+\frac{bc-ad}{c^{2}}h+O\left(h^{2}\right)\\
\left(\frac{1+ah^{2}}{1+bh^{2}}\right)^{c} & =1+ch^{2}(a-b)+O\left(h^{4}\right)
\end{align*}
we find 
\begin{align*}
\partial_{x}S_{\alpha,\beta}\left(x+h,x-h\right) & =\left(\frac{\left(1+O\left(h^{2}\right)\right)}{\left(1+O\left(h^{2}\right)\right)}\right)^{\frac{1}{\alpha-\beta}-1}\left(\frac{\left(1+h\frac{1}{3}x^{-1}\left(\alpha+\beta\right)+O\left(h^{2}\right)\right)}{2\left(1+hx^{-1}+O\left(h^{2}\right)\right)}\right)\\
 & =\left(\frac{1}{2}+\frac{\frac{2}{3}\left(\alpha+\beta\right)-2}{4\,x}\,h\right)+O\left(h^{2}\right)
\end{align*}
and hence \eqref{eq:Stolarsky-second-deriv}.
\end{proof}

\textcolor{red}{}%

\subsection{Approximation of potentials to get the SQRA mean\label{subsec:Approximation-of-potential}}

The aim of this section is to provide a class of potentials which
are easy to handle and which generate the SQRA-mean $S_{-1,1}(\pi_{0},\pi_{h})$
by $\pi_{\mathrm{mean}}=\left(\frac{1}{h}\int_{0}^{h}\pi^{-1}\right)^{-1}$.
Clearly, choosing the constant potential $V(x):=V_{c}:=-\log S_{-1,1}(\pi_{0},\pi_{h})$
we obtain right mean. Although this works for any means, this has
two drawbacks
\begin{enumerate}
\item The potential jumps and hence the gradient is somewhere infinite,
which means that at these points the force on the particles is infinitely
high which is not physical.
\item Approximating a general function by piecewise constants, on each interval
the accuracy is only of order $h$. However, approximating a function
by affine interpolation the accuracy is of order $h^{2}$ on each
interval (see below for the calculation).
\end{enumerate}
So we want to get a potential which may be used as a good approximation
(i.e. approximating of order $h^{2}$), is physical (i.e. continuous)
and generates the SQRA-mean. Note, that most considerations below
also work for other Stolarsky means. For simplicity we focus on the
SQRA mean $S_{-1,1}$.

\subsubsection{Approximation order for linear approximation}

Let us first realize that a linear interpolation provides an approximation
of order $h^{2}$. Let $V:[0,h]\rightarrow\R$ be a general $C^{2}$-potential.
We define with $V(0)=V_{0}$ and $V(h)=V_{h}$
\[
\tilde{V}(x)=V_{0}+\frac{V_{h}-V_{0}}{h}x.
\]
 Then one easily checks that 
\[
V(x)=V_{0}+\partial_{x}V(0)x+\frac{1}{2}\partial_{x}^{2}V(0)x^{2}+O(h^{3})
\]
and hence, 
\[
V(x)-\tilde{V}(x)=\left(\partial_{x}V(0)-\frac{V_{h}-V_{0}}{h}\right)x+\frac{1}{2}\partial_{x}^{2}V(0)x^{2}+O(h^{3}).
\]
 Clearly, we also have
\[
V_{h}=V_{0}+\partial_{x}V(0)h+\frac{1}{2}\partial_{x}^{2}V(0)h^{2}+O(h^{3})
\]
 which yields
\[
V(x)-\tilde{V}(x)=-\frac{1}{2}\partial_{x}^{2}V(0)hx+\frac{1}{2}\partial_{x}^{2}V(0)x^{2}+O(h^{3})=\frac{1}{2}\partial_{x}^{2}V(0)(x-h)x+O(h^{3})=O(h^{2}).
\]

\subsubsection{Definition of potentials $\hat{V}$ which generate the SQRA mean}

We consider a piecewise linear potential of the form
\[
\hat{V}(x)=\begin{cases}
\frac{V_{c}-V_{0}}{x_{1}}x+V_{0} & ,x\in[0,x_{1}]\\
V_{c} & ,x\in[x_{1},x_{2}]\\
\frac{V_{h}-V_{c}}{h-x_{2}}(x-x_{2})+V_{c} & ,x\in[x_{2},h]
\end{cases}.
\]
where $x_{1},x_{2}\in[0,h]$ are firstly arbitrary and $V_{c}=-\log S_{-1,1}(\pi_{0},\pi_{h})=\tfrac{1}{2}(V_{h}+V_{0})$
. The potential is clearly continuous. Then 
\[
\frac{1}{h}\int_{0}^{h}\e^{\hat{V(}x)}\d x=\frac{x_{1}}{h}\frac{\e^{V_{c}}-\e^{V_{0}}}{V_{c}-V_{0}}+\frac{x_{2}-x_{1}}{h}\e^{V_{c}}+\frac{h-x_{2}}{h}\frac{\e^{V_{h}}-\e^{V_{c}}}{V_{h}-V_{c}}.
\]
Introducing the ratios $\alpha=\frac{x_{1}}{h}$ and $\beta=\frac{h-x_{2}}{h}$
(which are in $[0,1/2]$) , we want to solve $\frac{1}{h}\int_{0}^{h}\e^{\hat{V}(x)}\d x=\e^{\tfrac{1}{2}(V_{h}+V_{0})}$.
Indeed, introducing the difference of the potentials
$\bar{V}=V_{h}-V_{0}$, we obtain
\[
\lambda=\frac{\alpha}{\beta}=\frac{\e^{\bar{V}/2}-\bar{V}/2-1}{\e^{-\bar{V}/2}+\bar{V}/2-1}\approx1+\frac{1}{3}\bar{V}+\frac{1}{18}\bar{V}^{2}.
\]
Hence, any value $\alpha,\beta$ satisfying this ratio generates a
potential with the SQRA-mean.

\subsubsection{Proof that the potential approximates an arbitrary potential of order
$h^{2}$}

Since the linear potentials approximates a general potential of order
$h^{2}$ it suffices to approximate the linear potential $\tilde{V}$
by $\hat{V}$. We show that there are $\alpha,\beta$ satisfying $\frac{\alpha}{\beta}=\lambda$,
such that $\|\hat{V}-\tilde{V}\|_{C([x_{i},x_{i+1}])}=O(h^{2})$.
The difference of $\hat{V}$ and $\tilde{V}$ is the largest at $x=x_{1}$
or $x=x_{2}$. We estimate both differences. We have 
\[
\tilde{V}(x_{1})=V_{0}+\frac{V_{h}-V_{0}}{h}x_{1}=V_{0}+\alpha\bar{V},\ \ \tilde{V}(x_{2})=V_{0}+\frac{V_{h}-V_{0}}{h}x_{2}=V_{0}+(1-\beta)\bar{V}.
\]
Hence we have to estimate 
\[
\Delta_{1}:=|V_{0}-V_{c}+\alpha\bar{V}|,\ \ \Delta_{2}:=|V_{0}-V_{c}+(1-\beta)\bar{V}|.
\]
In the case of SQRA, one possible choice for $\alpha,\beta$ is given by $\alpha+\beta=1$.
Then $\Delta_{1}=\Delta_{2}=|V_{0}-V_{c}+\alpha\bar{V}|=|V_{0}-V_{c}+\frac{\lambda}{1+\lambda}\bar{V}|=\frac{1}{1+\lambda}|(1+\lambda)(V_{0}-V_{c})+\lambda\bar{V}|$.
We have $V_{0}-V_{c}=-\bar{V}/2$, and hence 
\[
\Delta_{1}=\Delta_{2}=\frac{1}{1+\lambda}\frac{\bar{V}}{2}|\lambda-1|.
\]
One can check that $\lambda\approx1+\bar{V}/3$ and hence, $\Delta_{1}+\Delta_{2}\approx\frac{V^{2}}{6}\approx O(h^{2})$.

\newcommand{\etalchar}[1]{$^{#1}$}

\end{document}